\documentclass[11pt]{article}
\usepackage{pdfpages}

\usepackage{algorithm,enumerate}
\usepackage{algpseudocode}
\usepackage{bbm}
\usepackage{amsmath}
\usepackage[hidelinks]{hyperref}
\allowdisplaybreaks
\renewcommand{\emph}[1]{\textit{#1}}
\usepackage{enumerate,amsmath,amsthm,latexsym,amssymb}
\usepackage{color}\usepackage{graphicx}

\newcommand{\old}[1]{}

\newcounter{rot}%\addtocounter{rot}{1}, \therot

\newcommand{\ignore}[1]{}

\newcommand{\set}[1]{\left\{#1\right\}}

\newcommand\bfrac[2]{\left(\frac{#1}{#2}\right)}

\newtheorem{theorem}{Theorem}[section]

\newtheorem{lemma}[theorem]{Lemma}

\newtheorem{claim}[theorem]{Claim}

\newtheorem{question}[theorem]{Question}

\def\cD{{\mathcal D}}
\def\cE{{\mathcal E}}
\def\cF{\mathcal{F}}

\def\cK{{\mathcal K}}
\def\cH{{\mathcal H}}

\def\cM{\mathcal{M}}

\def\cP{\mathcal{P}}
\def\cS{{\mathcal S}}

\def\Pr{\mathbb{P}}

\def\bd{{\bf d}}
\def\bx{{\bf x}}

\def\g{\gamma}
\def\e{\epsilon}
\def\la{\lambda}

\def\bz{\overline{\zeta}}
\def\bm{\overline{m}}
\def\bY{\overline{Y}}
\def\bwmax{\overline{wmax}}
\def\bW{\overline{W}}

\def\z{\zeta}

\def\gr{\textsc{2-Greedy }}
\def\grs{\textsc{2-Greedy. }}
\def\grc{\textsc{2-Greedy, }}

\title{An improved lower bound on the length of the longest cycle in random graphs}
\author{Michael Anastos\footnote{Institute of Science and Technology Austria, Klosterneuburg, Austria. Email:michael.anastos@ist.ac.at.  This project has received funding from the European Union’s Horizon 2020 research and innovation
programme under the Marie Sk\l{}odowska-Curie grant agreement No 101034413
\includegraphics[width=5mm, height=3.5mm]{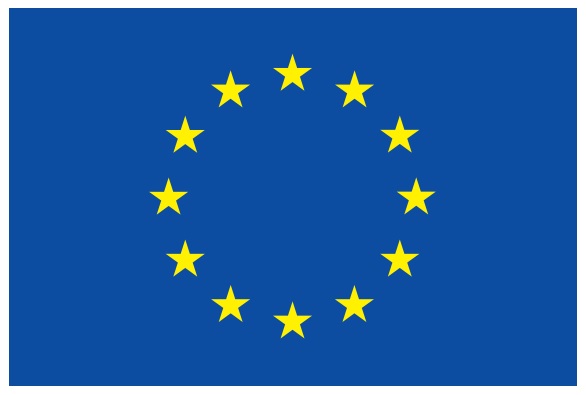}.}}

\begin{document}

\maketitle
\begin{abstract}
    We provide a new lower bound on the length of the longest cycle of the binomial random graph $G\sim G(n,(1+\e)/n)$ that holds w.h.p. for all $\e=\e(n)$ such that $\e^3n\to \infty$. In the case $\epsilon\leq \epsilon_0$ for some sufficiently small constant $\e_0$,  this bound is equal to $1.581\epsilon^2n$ which improves upon the current best lower bound of $4\epsilon^2n/3$ due to \L{}uczak.
\end{abstract}
\section{Introduction}
Let $G(n,p)$ denote the binomial random graph i.e. the random graph on $[n]$ where each edge appears independently with probability $p$.
For a graph $G$ denote by $L(G)$ the length of its longest cycle.
Erd\H{o}s \cite{erdos1975problems} conjectured that if $c>1$ then w.h.p.\footnote{We say that a sequence of events $\{\mathcal{E}_n\}_{n\geq 1}$ holds {\em{with high probability}} (w.h.p.\@ in short) if $\lim_{n \to \infty}\Pr(\mathcal{E}_n)=1-o(1)$.}  $L(G(n,c/n))\geq \ell(c)n$ where $\ell(c)>0$ is independent of $n$. This was proved by Ajtai, Koml\'os and Szemer\'edi \cite{ajtai1981longest} and in a slightly weaker form by  Fernandez  de la Vega \cite{de1979long} who proved that if $c>4\log 2$ then $L(G(n,c/n))=(1-O(c^{-1}))n$ w.h.p. Although this answers Erd\H{o}s's question and provides the order of magnitude of $L(G(n,c/n))$ for $c>1$ it leaves open the question of providing matching upper and lower bounds on $L(G(n,c/n))$. Since then, there has been an extended line of research trying to find such bounds. The corresponding results can be separated into 3 groups, primarily based on the method used. 

The first group of results builds upon an idea of Bollob\'as \cite{bollobas1982long}  who realised that for large $c$ one could find a large path/cycle w.h.p. by concentrating on a large subgraph with large minimum degree and demonstrating Hamiltonicity.
In this way Bollob\'as \cite{bollobas1982long},  Bollob\'as, Fenner and Frieze \cite{bollobas1984long} and Frieze \cite{frieze1986large} provided lower bounds on $L(G(n,c/n))$. Recently the author and Frieze were able to determine $L(G(n,c/n))/n$ for sufficiently large $c$ by examining a largest subgraph of $G(n,c/n)$ that can be Hamiltonian. Their result was then extended to  $c\geq 20$ by the author \cite{anastos2022note}. 

The second group of results analyses the Depth First Search (DFS) algorithm for identifying the connected components of a graph and argues that while doing so a long path is identified. In this group we find the result of Koml\'os and Szemer\'edi \cite{ajtai1981longest}.
%and of Fernandez de la Vega \cite{de1979long}. 
Enriquez, Faraud and M{\'e}nard determined the limiting shape of a the DFS-tree and consequentially the size of the longest root to leaf path that it identifies for $c>1$ \cite{enriquez2020limiting}. This path, in the regime $c=1+\epsilon$ where $\e$ is a sufficiently small constant has size $(1+o(1))\e^2n$ w.h.p. This result was extended to $\e=\omega(n^{-\frac{1}{3}+o(1)})$ by Diskin and Krivelevich  \cite{diskin2021performance}. Earlier Krivelevich and Sudakov showed that the DFS algorithm finds a path of length $\Theta(\epsilon^2 n)$ for $\epsilon=\omega( n^{-1/3}\log^{1/3}n)$ w.h.p. \cite{krivelevich2013phase}. 

In the final group we find \L{}uczak's result \cite{luczak1991cycles}. He
studied $L(G(n,c/n))$ in the regime $c=1+\epsilon$ where $\e=\epsilon(n)=o(1)$ is such that $\e^3n\to \infty$. By studying the kernel (defined shortly) of $G(n,(1+\e)/n)$, he proved that  $(1+o(1))4\epsilon^2 n/3\leq L(G(n,c/n))\leq (1+o(1))2\epsilon^2n$  w.h.p. The upper bound corresponds to the number of vertices of the $2$-core of $G(n,(1+\e)/n)$. Building upon his ideas, Kemkes and Wormald improved the upper bound to $1.7395 \epsilon^2 n$ \cite{kemkes2013improved}. Kim and Wormald announced the lower bound of $(1.5+o(1))\e^2n$ on $L(G(n,(1+\e)/n))$ which they derived by studying an adaptive version of the greedy Depth First Search algorithm applied to the kernel of $G(n,(1+\e)/n)$ equipped with the appropriate weights as described at the next section \cite{kemkes2013improved},\cite{kim2013}.
The kernel of a random graph $G$, denoted by $\cK(G)$, is the multigraph obtained by first generating the 2-core of $G$ (i.e. the maximal subgraph of $G$ of minimum degree 2), then contracting every maximal path whose internal vertices have degree 2 into a single edge and finally discarding any isolated vertices. By study the kernel of $G(n,p)$ we prove the following theorem.
\begin{theorem}\label{thm:main}
Let $G\sim G(n,(1+\epsilon)/n)$ where $\epsilon=\epsilon(n) \leq \epsilon_0$ for some sufficiently small $\epsilon_0>0$ and $\epsilon^3n\to \infty$. Then w.h.p. $L(G)\geq \alpha\cdot 4\e^2 n/3\geq 1.581\e^2 n$ where $\alpha>1.186$ is derived from a system of differential equations. 
\end{theorem}
Observe that our new bound reduces the gap between the published upper and lower bounds on $L(G(n,c/n))$ in the corresponding regime by a multiplicative factor of $0.5$. Our method for proving Theorem \ref{thm:main} can be extended to all $c>1$. Instead of doing so we use a weaker version of our arguments in order to extend \L{}uczak lower bound to all $c>1$ as follows.
\begin{theorem}\label{thm:weak}
Let $G\sim G(n,c/n)$ where $c>1$ is a constant. Let $n_2(G)$ denote the number of vertices of degree $2$ in the $2$-core of the giant component of $G$. In addition let $n(\cK(G))$ and $e(\cK(G))$ be the number of vertices and edges respectively of $\cK(G)$. Then w.h.p,
\begin{align}\label{eq:weak_bound}
L(G)\geq(1+o(1)) \bigg(1+ \frac{ n_2(G)}{e(\cK(G))}\bigg) n(\cK(G)).
\end{align}

\end{theorem}
When $G\sim G(n,c/n)$, $c=1+\e$,  $\e=\e(n)=o(1)$ and $\e^3n\to \infty$ we have that $n_2(G)=(1+o(1))2\e^2n$, $n(\cK(G))=o(n_2(G))$ and $n(\cK(G))=(1.5+o(1))e(\cK(G))$ w.h.p. In this regime Theorem \ref{thm:weak} implies that $L(G)\geq (1+o(1))4\e^2n/3$ thus recovering \L{}uczak's result.

We deduce theorems \ref{thm:main} and \ref{thm:weak} from the 
corresponding ones related to the uniform random multigraph $G^{\cM}(n,m)$ stated below. For positive integers $n,m$ we let $G^{\cM}(n,m)$ be the random multigraph  on $[n]$ with $m$ edges where every edge is chosen independently and uniformly at random from $[n]\times [n]$.
 
\begin{theorem}\label{thm:mainmult}
Let $G\sim G^\cM(n,n/2+s)$ where $s=s(n) \leq \epsilon_0 n/2$ for some sufficiently small $\epsilon_0>0$ and $s^3 n^{-2}\to \infty$. Then w.h.p. 
$L(G)\geq \alpha \cdot 16s^2/3n\geq 6.325s^2/n$ where $\alpha>1.186$ is derived from a system of differential equations. 
\end{theorem}

\begin{theorem}\label{thm:weakmult}
Let $G\sim G^\cM(n,n/2+s)$ where $s=s(n)=\Theta(n)$. Let $n_2(G)$ denote the number of vertices of degree 2 in the 2-core of the giant component of $G$. In addition let $n(\cK(G))$ and $e(\cK(G))$ be the number of vertices and edges respectively of $\cK(G)$. Then w.h.p,
\[L(G)\geq(1+o(1)) \bigg(1+ \frac{ n_2(G)}{e(\cK(G))}\bigg) n(\cK(G)) .\] 
\label{eq:weakthmmult}
\end{theorem}

To deduce theorems \ref{thm:main} and \ref{thm:weak} from theorems \ref{thm:mainmult} and \ref{thm:weakmult} respectively we note that when $s=O(n)$ the multigraph $G\sim G^\cM(n,n/2+s)$ is simple with some  probability $p(s)=\Omega(1)$. In addition the property of having a longest cycle of length at least $\ell$ is an increasing one for all $\ell=\ell(n)$.

\subsection{ \L{}uczak's lower bound}
We now give a sketch of  \L{}uczak's lower bound. 
Given $G\sim G(n,(1+\epsilon)/n)$ first generate the $2$-core of its giant component which we denote by $C_2(G)$ and then its kernel $\cK(G)$. Define a weight function $w:E(\cK(G))\mapsto \mathbb{N}^+$ as follows. For every edge $e$ of $\cK(G)$ let $w(e)$ be equal to the length of the maximal path in $C_2(G)$ which was contracted resulting to $e$. Note that we let $w_e=1$  if $e\in E(C_2(G))$. 
For a cycle $C$ in $\cK$ let $w(C)=\sum_{e\in C}w(e)$. Observe that for $k\geq 1$ a cycle $C$ in $\cK(G)$ of $w$-weight $k$ i.e. $w(C)=k$  corresponds to a cycle of length $k$ in $G$. Now, if $\epsilon=o(1)$ then the distribution of $\cK(G)$ is close to the distribution of a random 3-regular graph which is w.h.p. Hamiltonian (see \cite{robinson1992almost}). Note that the sum of weights in $\cK(G)$ is $n_2(G)+e(\cK(G))$ hence every edge of $\cK(G)$ has expected weight $1+n_2(G)/e(\cK(G)$. By linearity, a random Hamilton cycle in $\cK(G)$ has expected weight
$$\bigg(1+ \frac{n_2(G)}{e(\cK(G))}\bigg) n(\cK(G))=\frac{(1+o(1))4\epsilon^2 n}{3}.$$
At the above line line we used that w.h.p. (i) $C_2(G)$ spans $(1+o(1))2\epsilon^2 n$ vertices of degree 2 and (ii) $3n(\cK(G))=(1+o(1))e(\cK(G))$. By further studying the distribution of $w$ one can show that $w(H)\geq (1+o(1))4\epsilon^2 n/3$ w.h.p.

\subsection{Sketch of the proof of Theorem \ref{thm:main}}
A 2-matching $M$ of a graph $G$ is a set of edges such that every vertex in $G$ is incident to at most 2 edges of $M$. To prove Theorem \ref{thm:main} 
we first generate the kernel $\cK(G)$ of $G\sim G^\cM(n,m)$ and define the weight function $w$ on its edges. Then we peel from $\cK(G)$ a random set of edges which we put aside. We continue by implementing the \gr algorithm in order to construct a small collection of $w$-heavy vertex disjoint paths $\cP$ in the kernel of $G$ that cover its vertex set. Finally we utilise the edges that we put aside in order to merge the paths in $\cP$ into a cycle $C$ in $\cK(G)$ that uses the majority of the edges spanned by these paths. We use this later fact in order to argue that $C$ is $w$-heavy in $\cK(G)$ and therefore corresponds to a long cycle in $G$.

The \gr algorithm can be considered as an extension of the Karp-Sipser algorithm \cite{karp1981maximum}. The specific version of the \gr algorithm is taken from \cite{anastos2020hamilton} while variations and extensions of it have also been studied in \cite{anastos2021k}, \cite{bal2012greedy} and \cite{frieze2014greedy}. The \gr algorithm sequentially grows a 2-matching $M$. In parallel it removes edges and vertices from $G$. It always prioritises and matches vertices that are incident to $k\in \{0,1\}$ edges in $M$ and at most $2-k$ edges in $G$. If such vertices do not exist  then it adds to $M$ a $w$-heaviest edge. A more formal description of \gr is given at Section \ref{sec:2greedy}. To analyse the rate in which the weight of $M$ increases during the execution of \gr and subsequently the weight of the final matching we use the method of differential equations for dynamic concentration \cite{wormald1999differential}. 

\subsection{Organization of the paper} 
The rest of the paper is organised as follows. At Section \ref{sec:kernel} we study the kernel of a random multigraph. Thereafter at Section \ref{sec:2greedy} we present and analyse the \gr algorithm. We prove theorems  \ref{thm:mainmult} and \ref{thm:weakmult} at Section \ref{sec:proofmain}. 

\section{The Kernel of a random multigraph}\label{sec:kernel}
For $\gamma\in (0,1]$ we denote by $Geom(\gamma)$ the geometric random variable with probability of success $\g$ i.e. if $X\sim Geom(\g)$ then $\Pr(X=k)=\gamma(1-\g)^{k-1}$ for $k=1,2,...$. For this section we fix  $G\sim G^{\cM}(n,n/2+s)$ where $s^3n^{-2}\to \infty$ and $s=O(n)$. We also let $C_2(G)$ and $\cK(G)$ be the 2-core of the giant component of $G$ and the kernel of $G$ respectively. In addition for $k\geq 2$ we let $n_k(G)$ be the number of vertices of degree $k$ in $C_2(G)$. Finally we define the function $w:E(\cK(G))\mapsto \{1,2,...\}$ as follows. For $e\in E(\cK(G))$ let $w(e)$ be equal to the length of the path in $C_2(G)$ contracted to $e$.
\begin{lemma}\label{lem:kernel}
Conditioned on $n_2(G)$, $\sum_{k\geq 2} n_k(G)$ and $\sum_{k\geq 2} kn_k$  we have that $\cK(G)$ is distributed as a  random connected multigraph on $n'=\sum_{k\geq 3} n_k$ vertices with $m'=0.5\sum_{k\geq 3} kn_k$ edges and minimum degree 3. In addition the weights $\{w(e)\}_{e\in E(\cK(G))}$ are distributed as independent $Geom(\gamma)$ random variables conditioned or their sum being equal to $n_2+m'$.
\end{lemma}

\begin{proof}
Let $\gamma\in(0,1)$. Conditioned on $n_2(G)$, $\sum_{k\geq 2} n_k(G)$ and $\sum_{k\geq 2} kn_k(G)$ we have that $C_2(G)$ is distributed uniformly at random over all connected graphs on $\sum_{k\geq 2} n_k(G)$ vertices with  
\\$0.5\sum_{k\geq 2} kn_k(G)$ edges, $n_2(G)$ vertices of degree 2 and minimum degree 2, the reason being that each such graph is equally likely to be equal to $C_2(G)$ conditioned on the aforementioned quantities. Thus the pair $(\cK(G),w)$ is such that $\cK(G)$ is distributed as a random connected multigraph on $n'$ vertices with $m'$ edges and minimum degree 3 and $(w(e_1),....,w(e_{m'}))$ is (independently of the realisation of $\cK(G)$) chosen uniformly from the set $S=\{\bx\in (\mathbb{N}^+)^{m'} :\sum_{i=1}^m \bx_i = n_2+ m'\}$ the reason being there is $1$ to $1$ correspondence between such pairs $(\cK(G),w)$ and possible realizations of $C_2(G)$. Finally note that the joint distribution of $(X_1,...,X_{m'})$ where  $\{X_i\}_{i\in [m']}$ are i.i.d $Geom(\gamma)$ conditioned on their sum being equal to $n_2(G)+m'$ assigns the uniform measure to the elements of $S$.
\end{proof}
When we will later study $\cK(G)$ it would be in our interest to only consider the case where $\cK(G)$ does not span many loops and multiple edges. The following lemma enables us to do so.  Its proof is located at Appendix \ref{sec:app:lem:loops}.
\begin{lemma}\label{lem:loops}
Let $f(n)$ be any function of $n$ such that $\lim_{n\to \infty}f(n) =\infty$ and $G\sim G^{\cM}(n,m)$ where $m=O(n)$. Then w.h.p. $\cK(G)$ has fewer than $f(n)$ loops and multiple edges. In addition w.h.p. $\cK(G)$ does not span a multiple edge/loop of multiplicity larger than 2. 
\end{lemma}

At the proof of Theorem \ref{thm:mainmult}, instead of working with geometric random variables we work with $Exp_{\leq C}(1)$ random variables i.e. exponential random variables with mean 1 that are conditioned on being at most $C$ for some constant $C$, its probability density is given by the function $f(x)=e^{-x}/(1-e^{-C})$ for $x\in [0,C]$. To pass from the geometric to the exponential random variables we use the following lemma. Its proof is located at Appendix \ref{sec:app:lem:kernelexp1}
\begin{lemma}\label{lem:kernelexp1}
Let $0<\gamma=\gamma(n) < 0.01$ and $m'=m'(n)$ be such that $\lim_{n\to \infty} \g^{-1}(n)m'(n)=\infty$. Let $X_1,...,X_{m'}$ be independent $Geom(\gamma)$ random variables conditioned on their sum be equal to $\gamma^{-1}m'$. Let $C$ be a positive constant such that $e^{-C}\geq \g$ and $Y_1,Y_2,....,Y_{m'}$ be independent $Exp_{\leq C}(1)$ random variables. Then there exists a coupling between $(X_1,X_2,...,X_{m'})$ and $(Y_1,Y_2,...,Y_m')$ such that $\gamma^{-1}Y_i\leq  X_i$ for all $i\in  [m']$ w.h.p.
\end{lemma}

\section{The \gr  algorithm}\label{sec:2greedy}

\subsection{Random multigraphs with conditioned degree sequence}
Before stating the \gr algorithm we take some time to explain the random graph model that we will use to generate the graph at which we will apply \grs We use a variation on the pseudo-graph model of Bollob\'as and Frieze  \cite{bollobas1983matchings} and Chv\'atal \cite{chvatal1991almost} which we denote by $G_{n,m}(Y_1,Y_2,Y,Z_1,Z)$. Given $n,m \in \mathbb{N}$ and a sequence $\bx=(x_1,x_2,...,x_{2m})\in [n]^{2m}$ we define the multigraph $G_\bx$ by $V(G_\bx):=[n]$, $E(G_\bx):=\{\{x_{2j-1},x_{2j}\}:j\in [m]\}$. Thus $G_\bx$ is a graph on $n$ vertices with  $m$ edges. The degree of some vertex $v\in [n]$ with respect to the sequence $\bx$ is equal to the number of times it appears in $\bx$, i.e. $d_\bx(v)=|\{i: \bx_i=v, 1\leq i\leq 2m\}|$. Let $Y_1, Y_2,  Y,  Z_1$ and $Z$ be pairwise disjoint subsets of $[n]$. For $2m\geq |Y_1|+2|Y_2|+3|Y|+|Z_1|+2|Z|$ we let $\cS_{n,m}(Y_1,Y_2,Y,Z_1,Z)$ be the set of sequences $\bx \in [n]^{2m}$  such that 
\begin{align*}
    d_\bx(i)=
    \begin{cases}
    1&\text{ if }i\in Y_1\cup Z_1,
\\    2&\text{ if }i\in Y_2,
 \\   \geq 2&\text{ if } i\in Z,
 \\   \geq 3&\text{ if } i\in Y.
    \end{cases}
\end{align*}
We then define $G_{n,m}(Y_1,Y_2,Y,Z_1,Z)$ be the graph $G_\bx$ where $\bx$ is chosen uniformly at random from $\cS_{n,m}(Y_1,Y_2,Y,Z_1,Z)$. We also let $\cS_{n,m}^\cD(Y_1,Y_2,Y,Z_1,Z)$ be   set of degree sequences $\bd$ such that \begin{align*}
    \bd_i=
    \begin{cases}
    1&\text{ if }i\in Y_1\cup Z_1,
\\    2&\text{ if }i\in Y_2,
 \\   \geq 2&\text{ if } i\in Z,
 \\   \geq 3&\text{ if } i\in Y.
    \end{cases}
\end{align*}
Finally let $G_{n,m}^{\cM}(Y_1,Y_2,Y,Z_1,Z)$ be a multigraph chosen uniformly at random among the multigraphs with a degree sequence in $\cS_{n,m}^\cD(Y_1,Y_2,Y,Z_1,Z)$. Later on, when we will analyse \grc we will have to keep track of information like ``every vertex in $Z_2$ has been matched exactly once but still have degree at least $2$" or 
``every vertex in $Y_1$ has been matched exactly $0$ times so far and has degree $1$" e.t.c., hence the need of the $5$ distinct sets.

Observe that  $G_{n,m}^{\cM}(Y_1,Y_2,Y,Z_1,Z)$ assigns the uniform measure over multigraphs with a degree sequence in $\cS_{n,m}^\cD(Y_1,Y_2,Y,Z_1,Z)$ while $G_{n,m}(Y_1,Y_2,Y,Z_1,Z)$ does not. Indeed, if $G$ is such a multigraph with $l$ loops, $f$ edges of multiplicity $2$ and no edges of higher multiplicity then there are $(m!)2^m \cdot 2^{-f-l}$ sequences in  $\cS_{n,m}(Y_1,Y_2,Y,Z_1,Z)$ that correspond to the graph $G$. However as any multigraph may arise from at most $m!2^m$ sequences in $\cS_{n,m}(Y_1,Y_2,Y,Z_1,Z)$ the following lemma holds. 
\begin{lemma}\label{lem:transferdistribution}
Let $\cP$ be a graph property, $G\sim G_{n,m}(Y_1,Y_2,Y,Z_1,Z)$ and $G'\sim G_{n,m}^{\cM}(Y_1,Y_2,Y,Z_1,Z)$. Let $f(n)$ be any function such that $\lim_{n\to \infty} f(n) = \infty$. If $\Pr(G\text{ is connected})=1-o(1)$ and $\Pr(G\in \cP)=1-o(1)$ then $\Pr(G'\in \cP|\cE )=1-o(1)$ where $\cE$ is the event that $G'$ does not have a multiple edge with multiplicity larger than $3$, has at most $f(n)$ loops and multiple edges and it is connected.   
\end{lemma}
Recall that Lemma \ref{lem:loops} states that w.h.p. $\cK(G)$ satisfies the event $\cE$ thus the above lemma is good enough for us to work with. 
We finish this subsection by stating a lemma about the maximum degree of $G_{n,m}(Y_1,Y_2,Y,Z_1,Z)$. Its proof is located at Appendix \ref{sec:app:config}. There, we also further discuss the distribution of the degree sequence of $G\sim G_{n,m}(Y_1,Y_2,Y,Z_1,Z)$. 
\begin{lemma}\label{lem:maxdegree}
Let $Y_1, Y_2,  Y, Z_1, Z$ be pairwise disjoint subsets of $[n]$ and $m$ be such that  $|Y_1|+2|Y_2|+3|Y|+|Z_1|+2|Z|\leq 2m =O(n)$. Let $G\sim G_{n,2m}(Y_1,Y_2,Y,Z_1,Z)$. Then w.h.p. $\Delta(G)\leq \log n$. Thereafter if $Y=[n]$ then $G$ is connected w.h.p. 
\end{lemma}

\subsection{The description of \gr}

\gr will be applied to a (multi)-graph $G$ whose distribution is close  to $G_{n,2m}(\emptyset,\emptyset,[n],\emptyset,\emptyset)$. As the algorithm progresses, it makes changes to $G$ and generates a graph sequence $G_0=G \supset G_1 \supset ...\supset G_{\tau}=\emptyset$. In parallel, it grows a $2$-matching $M$. We let $M_0=\emptyset, M_1,M_2,..., M_{\tau}=M$ be the sequence of $2$-matchings that are generated. For $0\leq t\leq \tau$ by \emph{``at time/step $t$"}   we refer to the $t^{th}$ iteration of the while-loop at line 3 of the description of \grc we also define the following quantities and sets.
\begin{itemize}
\item $G_t$, the graph at the beginning of step $t$, 
\item $m_t:=|E(G_t)|$,
\item $d_t(v):=d_{G_t}(v)$, for $v\in V$,
\item $d_{M_t}(v):=|\set{ e \in M_t:v \in e} |+|\{\text{loops in $M_t$ appended at $v$}\}|$, for $v\in V$,
\item $Y_{j}^t:=\set{v \in [n]: d_{M_t}=0,\,d_{t}(v)=j}$, $j\geq 1$, i.e. the set of vertices of degree $j$ in $G_t$ that are incident to $0$ edges in $M_t$,
\item $Z_{j}^t:=\set{v \in [n]: d_{M_t}=1,\,d_{t}(v)=j}$, $j\geq 1$, i.e. the set of vertices of degree $j$ in $G_t$ that are incident to $1$ edge in $M_t$,
\item $Y^t:=\cup_{j\geq 3} Y_{j}^t$ and $Z^t:=\cup_{j\geq 2} Z_{j}^t$,
\item $D^t=Z_1\cup Y_1 \cup Y_2$, the set of ``dangerous" vertices,
\item $\zeta_t = |Z_1|+|Y_1|+2|Y_2|$.
\end{itemize}
\noindent
Observe that if $v\in G_t$ satisfies $v\in  D^t$  then there are $k\in \{0,1\}$ edges incident to $v$ in $M_t$ and at most $2-k$ edges incident  to $v$ in $G_t$. Thus in the final $2$-matching $M$, $v$ can be incident to at most $2$ edges. In addition, if in the future an edge incident to $v$ in $G_t$ is removed but not added to the matching then  $v$ will be incident to at most $1$ edge in $M$. Therefore $D^t$ consists of {\emph dangerous vertices} whose edges we would like to add to the matching as soon as possible in order to avoid ``accidentally" deleting them. $\zeta_t$ bounds the number of those edges. \gr tries to match the vertices in $D^t$ first.  If $D^t$ is empty then \gr adds to $M_t$ an edge of maximum weight.
\begin{algorithm}[H]
\caption{{\gr}}
\begin{algorithmic}[1]
\\ Input: a graph $G_1=G$ and a function $w:E(G)\mapsto[0,\infty)$.
\\ $t = 0$, $M_0=\emptyset$.
\While{$E(G_t)\neq \emptyset$}
{
\If{$D^t \neq \emptyset$}
\\\hspace{10mm} Choose a  vertex $v_t \in D^t$  uniformly at random.
\\\hspace{10mm} Choose an edge $e_t=v_tu_t$ in $G_t$ incident to $v_t$ uniformly at random. 
\Else
\\\hspace{10mm} Choose an edge $e_t=v_tu_t$ in $G_t$ among those that maximise $w(e)$ uniformly at random. 
\EndIf 
\\\hspace{5mm} Set $M_{t+1} = M_t \cup\{\{v_t,w_t\}\}$
\\\hspace{5mm} Delete the edge $e_t$ from $G_t$.
\\\hspace{5mm} For $z\in \{v_t,u_t\}$ if $z$ is incident to at least 2 edges or to a loop in $M_{t+1}$ then delete all the edges incident to $z$ in $G_t$.
\\\hspace{5mm} Delete all the isolated vertices in $G_{t}$ and let $G_{t+1}$ be the resultant graph.
\\\hspace{5mm} Set $t = t+1$. }
\EndWhile
\\ Set  $\tau = t$.
\end{algorithmic}
Remove any loops and multiple edges from $M_\tau$.
\\Output $M_\tau$.
\end{algorithm}

We let $e_t=v_tu_t$ and $R_t$ be the set of edges deleted at step $t$. We also denote by $\cH_t$ the information consisting of the actions taken by \gr by time $t$ and the sets  $Y_1^i,Y_2^i,Y^i,Z_1^i,Z^i$, $i\leq t$. Finally we set $M_i=M_\tau$ for $i\geq \tau$. The performance of $\gr$ in settings of interest is given by the following 2 lemmas. 
\begin{lemma}\label{lem:performance}
Let  $Y_0\sqcup Y_1\sqcup Y_2 \sqcup Y$ be a partition of $[n]$ such that $|Y_2|=\Theta(n^{0.95})$, $|Y_0\cup Y_1|=O(n^{0.9})$ and $2m \geq |Y_1|+2|Y_2|+ 3|Y|$. Let $G\sim G_{n,m}(Y_1,Y_2,Y,\emptyset,\emptyset)$ and $V_2$ be the set of neighbors of $Y_2$ in $G$. To each edge $e\in E(G)$ assign a weight $w(e)\sim Exp_{\leq 20}(1)$. Then execute \gr with input $G,w$ and let $M_i,i\geq 0$ be the generated $2$-matchings, $M=M_\tau$ the output matching and $\tau'=(1-\log^{-1}n)n$. There exists $0<\epsilon_1<0.01$ such that if $ m \leq (1.5+\epsilon_1^3)n$ then w.h.p. the following hold.
\begin{itemize}
    \item[(a)] $\sum_{e\in M_\tau} w(e)\geq (\alpha+\beta)n$ where $\alpha>1.186$ is derived from a system of differential equations (see equations \eqref{eq:Y}-\eqref{eq:Wt} and \eqref{eq:estimate}) and $\beta>0 $ is a sufficiently small constant that is independent of $\alpha$.
    \item[(b)] $M_\tau$ has size $n-O(n^{0.9})$ and spans $O(n^{0.9})$ components.
    \item[(c)] $|V_2\cap V(G_t)| \geq 0.5(1-t/n)^{10}|V_2|$ for $t\leq \tau'$.
    \item[(d)] Every path spanned by $M_{\tau}$ of length at least $n^{0.06}$ with more than $n^{0.06}/2$ edges in $M_{\tau'}$ spans a vertex in $Y_2$.
\end{itemize}
\end{lemma}
Parts (a) and (d) are proven at subsections \ref{subsec:(a)} and \ref{subsec:(c)} respectively. The proof of parts (b) and (c) are located at Appendix \ref{sec:app:auxlemmas}. The proof of the following lemma is identical to the proof of part (b) of Lemma \ref{lem:performance}, hence it is omitted.  
\begin{lemma}\label{lem:performance2}
Let $\epsilon,\gamma>0$  $m \geq 1.5 n$ and  $G\sim G_{n,m}(\emptyset, \emptyset, [n],\emptyset,\emptyset)$. To each edge $e\in E(G)$ assign a weight $w(e)=1$. Then execute \gr with input $G,w$ and let $M$ be the output matching. Then w.h.p. $M$ has size $n-O(n^{0.9})$ and spans $O(n^{0.9})$ components.
\end{lemma}
The first thing that we need for the analysis of \gr is to characterize the distributions of $G=G_1,G_2,...,G_\tau$. It turns out  that $G_t$ exhibits a Markovian behaviour as described by the following lemma. 
\begin{lemma}\label{lem:uniformity}[Lemma 5.3 of \cite{frieze2014greedy}]
If $G_{0}\sim G_{n_0,m_0}(Y_1^0,Y_2^0,Y^0,Z_1^0,Z^0)$ then for $t\geq 1$, conditioned on $\cH_{t}$  we have that $G_{t}\sim G_{n_{t},m_{t}}(Y_1^{t},Y_2^{t},Y^{t},Z_1^{t},Z^{t})$.
\end{lemma} 
Given the above lemma and the description of \gr what we are now missing in order to start the analysis of \gr are estimates on  probabilities like $p_{ij}=\Pr($a random edge of $G_t$ joins a vertex of degree $d_i$ to a vertex of a degree $d_j$) and $p_i=\Pr($ a random neighbor of $v_1\in V(G_t)$ has degree $d_i)$. Now given $\bd(G_t)$, the degree sequence of $G_t$, and assuming that $\Delta(G)\leq \log n$ such probabilities are easy to calculate as each sequence $\bx$ with degree sequence $d_\bx=\bd(G_t)$ corresponds to a partition into 2 elements sets of a $\sum_{v \in V(G_t)}d_t(v)$-element multi-set in which the element $v$ appears $d_t(v)$ times. So for example  if $i\neq j$ then $p_{ij}$ equals to $(i\nu_i/2m_t) (j\nu_j/2m_t)+o(1)$ where $\nu_i$ equals the number of vertices of degree $i$, $i\geq 0$ in $\bd(G_t)$. This is because $(i\nu_i/2m_t) (j\nu_j/2m_t)+o(1)$ equals to the ratio of the number of multigraphs with degree sequence resulting by reducing the degrees of a  vertex of degree $i$ and a vertex of degree $j$ by $1$ over the number of multigraphs with degree sequence $\bd(G_t)$. Fortunately for us, we note the following. Initially $|Y^0|=(1+o(1))n$ and $G_1$ has $(1.5+\e_1^3)n$ edges, thus all but $(\e_1^3+o(1))n$ vertices lie in $Y^0$ and have degree 3. This implies that at step $t$ at most $(\e_1^3+o(1))n$ vertices lie in $Z^t$ and have degree larger than $2$ or lie in $Y^t$ and have degree larger than $3$. This property of $\bd(G_t)$ suffices and no further investigation of the distribution of $\bd(G_t)$ is needed for the purposes of proving part (a) of Lemma 3.2. We further discuss the distribution of the degree sequence of $G\sim G_{n,m}(Y_1,Y_2,Y,Z_1,Z)$ at Appendix \ref{sec:app:config}.

\subsection{Proof of Lemma \ref{lem:performance} (a)}\label{subsec:(a)}
We start by introducing some notation that we use through this subsection. 
For $t\leq \tau$ we let $$ wmax_t=\max\{w(e):e\in E(G_t)\},\hspace{5mm} W_t=\sum_{i=0}^t w(e_i)\hspace{5mm}\text{ and }\hspace{5mm} p_t=3|Y_3^t|/2m_t.$$ In addition for $t>0$ we let 
$$\bz_t=\sum_{i=t}^{t+\log^2 n-1} \mathbb{I}(\z_i>0)).$$ 
We also define the stopping times $\tau_1,\tau_2$ by, 
$$\tau_1=\min\{t: \z_t=0\} \text{ and }\tau_2=\min\{t: |Y_3^t| \leq \epsilon_1  n\}.$$
Finally, for $t \geq 0$ we let $\cE_t$ be the event that $\tau_1\leq \max\{ 200\e^{-2}_1\z_0,\log^2 n\} = O(n^{0.95}), \Delta(G)\leq \log n$ and $\zeta_i\leq 200\e^{-2}_1\log n$ for $i\in [\tau_1, \tau_2+\log^2 n]\cap [0,t]$.

To prove Lemma 3.2 (a) we use the differential equation method for dynamic concentration \cite{wormald1999differential}. The exact result that we use is Theorem \ref{thm:diff}, it is taken from \cite{warnke2019wormald} and stated at Appendix \ref{sec:app:diffmethod}. We use it in order to track the evolution of $|Y^t|,m_t,wmax_t$ and $W_t$ and subsequently provide a lower bound on $W_\tau=\sum_{e\in M_\tau}w(e)$ that holds w.h.p.

Upon calculating the 1-step change of $Y^t$ and $m_t$ conditioned on $\cH_t$, a step needed for applying the differential equation method, one realizes that the corresponding expressions involve the discrete $\cH_t$-measurable random variable $\mathbb{I}(\zeta_t>0)$ in a manner that makes it difficult to apply Theorem \ref{thm:diff} directly in any meaningful manner. To deal with this, later when we will apply the differential equation method, instead of using the 1-step changes directly we use the $\log^2 n$-steps changes and show that during intervals of size $\log^2 n$ we may approximate the parts of the expressions in interest involving $\mathbb{I}(\zeta_t>0),...,\mathbb{I}(\zeta_{t+\log^2 n}>0)$ by an expression that involves $Y^t$ and $m_t$ only. It is also worth mentioning that we  track the increase of $W_t$ from the first time $\z_t$ hits $0$ till $|Y_3^t| \leq \epsilon_1^3 n$ corresponding to $\tau_1$ and $\tau_2$ respectively.

We start by calculating 1-step changes conditioned on the event $\cE_t$. Lemma \ref{lem:probei} (stated later at this subsection) states that $\cE_t$ occurs  for $t\leq \tau_2$ w.h.p.
\begin{lemma}\label{lem:onestep}
\begin{align}\label{eq:1changez}
    \mathbb{E}(\z_{i+1}+\z_i|\cH_i,\cE_t) = -\mathbb{I}(\z_t>0)+(2-\mathbb{I}(\z_t>0)) \sum_{i\geq 2}\frac{i |Z_i^t|}{ 2m_t}\bigg(  \frac{2|Z_2^t|}{2m_t}+2\frac{3|Y_3^t|}{2m_t} \bigg) + O\bfrac{1+\z_t}{2m_t}.
\end{align}
\begin{align}\label{eq:1expchageY}
\mathbb{E}\big(|Y^{t+1}|-|Y^t|\big|\cH_t, \cE_t\big)= - (2-\mathbb{I}(\z_t>0))\bigg( \sum_{i\geq 3}\frac{i |Y_i^t|}{ 2m_t}+  \sum_{i\geq 2}\frac{i |Z_i^t|}{ 2m_t} \frac{(i-1) 3 |Y_3^t|}{ 2m_t}\bigg)+ O\bfrac{1+\z_t}{2m_t}. 
\end{align}
\begin{align}\label{eq:1changem_t}
\mathbb{E}(2m_{t+1}-2m_t|\cH_t, \cE_t) = -2 - (2-\mathbb{I}(\z_t>0)) \sum_{i\geq 2}\frac{i |Z_i^t|}{ 2m_t}\cdot 2(i-1) + O\bfrac{1+\z_t}{2m_t}.
\end{align}
\end{lemma}
\begin{proof}
We start by justifying \eqref{eq:1changez}, for that consider the step $t$ of \grs By adding $e_t$ to $M_t$ and removing it from $G_t$, $\z_t$ decreases in expectation by $\mathbb{I}(\z_t>0)+ \Pr(u_t \in D^t)+ \Pr(v_t\in D^t \wedge \z_t=0)$. The later two terms equal to $O(\z_t/m_t)$. Thereafter for $z\in \{v_t,w_t\}\setminus D^t$ if $z$ belongs to $Z^t$ then $z$ is removed from $G_t$, every neighbor $w$ of $z$ in $Z_2^t$ and $Y_3^t$ for which $wz$ is not a multiple edge, enters $Z_1^t$ and $Y_2^t$ respectively, resulting to the increase of $\z_t$ by $1$ and $2$ respectively. $i|Z_i^t|/2m_t$ and $j|Y_j^t|/2m_t$ account for the probabilities that a vertex chosen proportional to its degree belongs to $Z_i^t$ and $Y_j^t$ respectively. The $O((1+\z_t)/2m_t)$ additive factor in \eqref{eq:1changez}-\eqref{eq:1changem_t},  accounts for the unlikely events that we identify a loop or a multiple edge during this step or $u_t\in D^t$ or $\z_t=0$ and $v_t\in D^t$.  

Similarly for \eqref{eq:1expchageY} for every vertex $z\in \{v_t,w_t\}\setminus D^t$, if $z\in Y^t$ then $|Y^t|$ is decreased by $1$. 
Else if $z\in Z^t$ then $z$ is removed from $G_t$ and $|Y^t|$ decreased by the number of neighbors of $z$ in $Y^t_3\setminus \{v_t,u_t\}$.

For \eqref{eq:1changem_t}, $2m_t$ is initially decreased by $2$ (as $e_t$ is removed from $G_t$) and then for every vertex $z\in(Z\cap \{v_t,w_t\})\setminus D^t$ that is not incident to a loop or a multi-edge, $2m_t$ is further decreased by $2(d_t(z)-1)$ (as $z$ and the edges incident to it are removed from $G_t$).
\end{proof}
We will use the next lemma to deal with the terms $\mathbb{I}(\z_t>0)$ appearing in \eqref{eq:1changez}-\eqref{eq:1changem_t} later on. Recall  $p_t=3|Y_3|/2m_t$ and $\bz_t=\sum_{i=t}^{t+\log^2 n-1} \mathbb{I}(\z_i>0))$.
\begin{lemma}\label{lem:probei}
W.h.p. the event $\cE_t$ occurs for $t\in [\tau_1,\tau_2]$. In addition,
\begin{align}\label{eq:changezt}
    \bigg|\mathbb{E}\bigg(\bz_t|\cH_t,\cE_t) -\frac{(2-2p_t^2)\log^2 n}{2-p_t^2} \bigg) \bigg| =O(\epsilon_1^2\log^2 n) 
    \text{ for }t\in [\tau_1,\tau_2).
\end{align}
\end{lemma}
\begin{proof} We start by proving 2 claims. 
\begin{claim}
For $t\geq 0$ and $0\leq i\leq \tau-t$,
\begin{equation}\label{eq:bounds}
|\z_{t+i}-\z_t|, \big||Y^{t+i}|-|Y^t|\big|, |2m_{t+i}-2m_t|\leq 4i\Delta(G)    
\end{equation}
\end{claim}
\begin{proof}
The above claim follows from the fact that at step $j$ only the edges incident to $\{v_j,u_j\}$ are deleted from $G_j$. Therefore, $|Y^{j+1}\setminus Y^j| \leq |N_{G_j}(v_j)\cup N_{G_j}(u_j)|\leq 2\Delta(G)$, $2m_{j+1}-2m_j\leq d_j(v_j)+ d_j(u_j)\leq 2\Delta(G)$ and $|\z_{j+1}-\z_j|\leq 2|D^{j+1}\triangle D^j| \leq 2(d_j(v_j)+ d_j(u_j))\leq 4\Delta(G)$.  
\end{proof}

\begin{claim}
For $t\leq \tau_2+\log^2 n$ if $\zeta_t\leq \e_1^3 n$ and $\Delta(G)\leq \log n$ then, 
\begin{align}\label{eq:probabilitiesapprox}
\sum_{i\geq 3}\frac{2|Z^t_i|}{2m_t}+ \sum_{i\geq 4}\frac{3|Y^t_i|}{2m_t} \leq \epsilon_1^2 
& \text{\hspace{15mm} and} &1-\epsilon^2_1\leq \frac{2|Z^t_2|}{2m_t}+ \frac{3|Y^t_2|}{2m_t} \leq 1.
\end{align}
\end{claim}
\begin{proof}
 $\sum_{i\geq 3} i|Z^t_i|+ \sum_{i\geq 4}i|Y^t_i|$ counts the number of pairs $(e,v)$ where $e$ is an edge of $G_t$, $v$ is an endpoint of $e$ and $v\in (Z^t\cup Y^t)\setminus (Z_2^t\cup Y_3^t)$. Observe
that a vertex $v$ belongs to $(Z^t\cup Y^t)\setminus (Z_2^t\cup Y_3^t)$ only if it belongs to $(Z^0\cup Y^0)\setminus (Z_2^0\cup Y_3^0)$. Therefore,
\begin{align*}
    &\sum_{i\geq 3} i|Z^t_i|+ \sum_{i\geq 4}i|Y^t_i|\leq \sum_{i\geq 3} i|Z^0_i|+ \sum_{i\geq 4}i|Y^0_i| \leq (1.5+\epsilon_1^3)n-3|Y_3^0| =(\epsilon_1^3+o(1))n.
\end{align*}
For $t\leq \tau_2+\log^2 n$, using \eqref{eq:bounds} and $\Delta(G)\leq \log n$ we have that 
$$2m_t\geq 3|Y_3^{\tau_2+\log^2 n}|\geq
3|Y_3^{\tau_2-1}|-4(1+\log^2 n)\Delta(G) = (3+o(1))\epsilon_1 n.$$ Thus,
\[\sum_{i\geq 3}\frac{i|Z^t_i|}{2m_t}+\sum_{i\geq 4}\frac{i|Y^t_i|}{2m_t} \leq \frac{(\epsilon^3_1+o(1))N}{(3+o(1))\e_1 N} \leq \e_1^2 \]  
Similarly if $\zeta_t\leq \e^3_1 n$ then,
\[ \frac{2|Z^t_2|}{2m_t}+ \frac{3|Y^t_2|}{2m_t} =1-\frac{\z_t + \sum_{i\geq 3}i|Z^t_i|+\sum_{i\geq 4}i|Y^t_i|}{2m_t} \in [1-\e^2_1,1].\]
\end{proof}
Equation \eqref{eq:1changez} implies,
\begin{align}\label{eq:weakboundz}
   \mathbb{E}(\z_{i+1}+\z_i|\cH_i, \cE_t) \leq  -\mathbb{I}(\z_t>0)+(2-\mathbb{I}(\z_t>0))\bigg(1-\frac{3|Y^t_3|}{2m_t}\bigg)\bigg( 1+\frac{3|Y_3^t|}{2m_t} \bigg) +o(1).
\end{align}
In particular, if $\z_t>0$ then as $|Y_3^t|\geq \e_1 n+O(\log^3 n)$ and $m_t\leq m_0\leq (1.5+\e_1^3)n$ for $t\leq \tau_2+\log^2 n$,
\begin{align*}
    \mathbb{E}(\z_{i+1}+\z_i|\cH_i, \cE_t) \leq  -1+\bigg(1-\frac{3\e_1 n}{2(1.5+\e^3_1) n)}\bigg)\bigg( 1+\frac{3\e_1 n}{2(1.5+\e^3_1) n)} \bigg) \leq -0.7\e_1^2.
\end{align*}
Lemmas \ref{lem:maxdegree}  and \eqref{eq:bounds} imply that $|\z_t-\z_{t+1}|\leq 4\log n$ for $t\geq 0$ w.h.p. Initially $\z_0\leq |Y_1^0\cup Z_1^0|+2|Y_2^0|=O(n^{0.95})$. Hence the Azuma-Hoeffding inequality implies that $\z_t$ reaches 0 in the interval $[0,\max\{200\e^{-2}\z_0,\log^2 n\}]$, hence $\tau_1\leq \max\{200\e^{-2}\z_0,\log^2 n\}=O(n^{0.95})$ w.h.p. Similarly, for $t\in [t_1,t_2+\log^2 n]$ we have that $\z_t\geq 200 \e_1^{-2} \log n$ only if there exists $t'\in [t_1,t_2]$ such that $\z_{t'}=0$ and  $\z_{t''}>0$ for $t''\in [t',t-1]$. Once again, by taking union bound over $\tau_1\leq t'\leq t\leq \tau_2+\log^2 n \leq n+\log^2 n$, the Azuma-Hoeffding inequality implies that $\z_t\leq 200 \e_1^{-2} \log n$ for $t\in [\tau_1,\tau_2+\log^2 n]$ w.h.p. Therefore the event $\cE_t$ occurs for $0\leq t\leq \tau$ w.h.p. 

By repeating the above argument one can show that w.h.p. for $t\in [\tau_1,\tau_2]$,
\begin{equation}\label{eq:333}
-200\e_1^{-2}\log n\leq  \mathbb{E}(\z_{t+\log^2 n}-\z_t|\cH_t,\cE_t)\leq 400\e_1^{-2}\log n.
\end{equation}     
\eqref{eq:1changez} and \eqref{eq:probabilitiesapprox} imply that for $0\leq i<\log^2 n$
\begin{align}\label{eq:lem341}
    \mathbb{E}&(\z_{t+i+1}-\z_{t+i}|\cH_t,\cE_t)= \mathbb{E}(\mathbb{E}(\z_{t+i+1}+\z_{t+i}|\cH_{t+i})|\cH_i,\cE_i) 
\nonumber    \\&=\mathbb{E}\bigg( -\mathbb{I}(\z_{t+i}>0)+(2-\mathbb{I}(\z_{t+i}>0))\bigg(1-\frac{3|Y^{t+i}_3|}{2m_{t+i}}\bigg)\bigg( 1+\frac{3|Y_3^{t+i}|}{2m_{t+i}} \bigg)\bigg|\cH_t,\cE_t\bigg)+O(\e^2_1)
\nonumber    \\&=\mathbb{E}\bigg( -\mathbb{I}(\z_{t+i}>0)+(2-\mathbb{I}(\z_{t+i}>0))\bigg(1-\frac{3|Y^{t}_3|}{2m_{t}}\bigg)\bigg( 1+\frac{3|Y_3^{t}|}{2m_{t}} \bigg)\bigg|\cH_t,\cE_t\bigg)+O(\e^2_1)
\nonumber    \\& = 2\bigg(1-\bfrac{3|Y^{t}_3|}{2m_{t}}^2\bigg) - \bigg(2-\bfrac{3|Y_3^{t}|}{2m_{t}}^2\bigg)\mathbb{E}( \mathbb{I}(\z_{t+i}>0)|\cH_t,\cE_t) +O(\e^2_1).
\nonumber \\& = 2(1-p_t^2) - (2-p_t^2)\mathbb{E}( \mathbb{I}(\z_{t+i}>0)|\cH_t,\cE_t) +O(\e^2_1) \nonumber
\\&= (2-p_t^2)\bigg[\frac{2-2p_t^2}{2-p_t^2} - \mathbb{E}( \mathbb{I}(\z_{t+i}>0)|\cH_t,\cE_t)  +O(\e^2_1) \bigg].
\end{align}
\eqref{eq:changezt} follows from summing \eqref{eq:lem341} over $i\in\{0,1,...,\log^2 n-1\}$ and then using \eqref{eq:333} to provide upper and lower bounds on the resultant expression.
\end{proof}

We now define the variables $\bY_i,\bm_i, \bwmax_i$ and $\bW_i$ as follows. For $i\geq 0$ let $i_n=i\log^2 n$ and
$$ \bY_i=\frac{|Y^{\tau_1+i_n}|}{\log^2 n},\hspace{5mm} \bm_i=\frac{m_{\tau_1+i_n}}{\log^2 n},\hspace{5mm} \bwmax_i= \frac{n}{\log^2 n}\cdot wmax_{\tau_1+i_n} \text{ and } \bW_i=\frac{W_{\tau_1+i_n}-W_{\tau_1}}{\log^2 n}.$$
We also define the stopping time $T$ by 
$$T=\min\set{i:\tau_1+i_n \geq \tau_2}.$$

\begin{lemma}\label{boundedness}[Boundedness-Hypothesis] W.h.p. for $t\in [\tau_1, \tau_2)$, 
\begin{equation}\label{eq:bounds2}
    |\bY_{t+1}-\bY_t|,|2m_{t+1}-{2m_t}|, |\bwmax_{t+1}-\bwmax_t|,|\bW_{t+1}-\bW_t| \leq 4\log^6 n.
\end{equation}
\end{lemma}
\begin{proof}
In the event $\cE_0$ \eqref{eq:bounds} implies that  $|\bY_{t+1}-\bY_t|,|2m_{t+1}-{2m_t}|\leq 4\log^3 n$ for $t\geq 0$. Thereafter $|\bW_{t+1}-\bW_t\bigg|\leq (\max_{e\in E(G_0)} w(e))\log^2 n\leq 20\log^2 n$. $|\bwmax_{t+1}- \bwmax_t|>4\log^6 n$ only if the weight of fewer than $\log^2 n$  edges in $G_t$ lies in the interval $I=[wmax_t-(4\log^4n)/n,wmax_t]$.
Conditioned on $\bwmax_t$, hence on $wmax_t$, for $e\in E(G_t)$ the weight $w(e)$ is an independent $Exp_{\leq wmax_t}(1)$ random variable, thus it belongs to $I$ independently with probability 
$$q \geq e^{-wmax_t+\frac{4\log^4n}{n}}-e^{-wmax_t} = e^{-wmax_t}(e^{-\frac{4\log^4n}{n}}-1) \geq \frac{e^{-20}\log^4n}{n}.
$$
In addition $G_t$ spans at least $|Y^t|\geq \e_1n$ edges for $t\leq \tau$. Therefore,
$$\Pr(\exists t<T: |\bwmax_{t+1}- \bwmax_t|>4\log^6 n) \leq n  \binom{\e_1n}{\log^2n} (1-q)^{\e_1n-\log^2 n}=o(n^{-1}).$$
\end{proof}

\begin{lemma}\label{lem:trend}[Trend-Hypothesis] For $t<T$ the following hold.
 \begin{align}\label{eq:expchageY}
 \mathbb{E}(\bY_{t+1}-\bY_t|\cH_t,\cE_t)= - \frac{2p_t(2-p_t) }{2-p_t^2}+ O(\epsilon^2_1),
\end{align}
\begin{align}\label{eq:changem_t}
\mathbb{E}(2\bm_{t+1}-2\bm_t|\cH_t) = -\frac{8-4p_t-2p_t^2}{2-p_t^2}+ O(\epsilon^2_1),
\end{align}
\begin{align}\label{eq:changemax}
\mathbb{E}(\bwmax_{t+1}-\bwmax_t|\cH_t) = -\frac{\bigg(e^{\frac{\bwmax_t\log^2 n}{n}}-1\bigg) p_t^2}{ \bfrac{\bm_t\log^2 n}{n}(2-p_t^2)}+ O(\e^2_1),
\end{align}
and
\begin{align}\label{eq:changeWt}
\mathbb{E}(\bW_{t+i}-\bW_t|\cH_t) =\frac{p_t^2}{2-p^2_t} \bfrac{\bwmax_t\log^2 n}{n} +\bfrac{2-2p_t^2}{2-p_t^2}\bigg(1-e^{-\frac{\bwmax_t\log^2 n}{n}}\bigg)^2+ O(\epsilon^2_1).
\end{align}
\end{lemma}
\begin{proof}

The derivation of \eqref{eq:expchageY}-\eqref{eq:changem_t} is obtain is a manner similar to the derivation of \eqref{eq:changezt} from \eqref{eq:1changez} and using \eqref{eq:changezt} for approximating $\sum_{i=0}^{\log^2 n-1}\mathbb{I}(\z_{t+i}>0)$ by $(2-p^2_t)/(2-p_t^2)$  whenever it appears. For example, \eqref{eq:1expchageY}, \eqref{eq:changezt}, \eqref{eq:bounds}  and \eqref{eq:probabilitiesapprox}  imply,
 \begin{align*}
 \mathbb{E}(\bY_{t+1}-\bY_t|\cH_t,\cE_t)&= 
  - \bigg(2-\frac{2-2p_t^2}{2-p_t^2}\bigg)\bigg( \frac{3 |Y_3^t|}{ 2m_t}+  \frac{2 |Z_i^t|}{ 2m_t}\frac{3 |Y_3^t|}{ 2m_t}\bigg)+ O(\epsilon^2_1),
 \\& =  - \bigg(2-\frac{2-2p_t^2}{2-p_t^2}\bigg)( p_t+(1-p_t)p_t)+ O(\epsilon^2_1)=  - \frac{2p_t(2-p_t)}{2-p_t^2}+ O(\epsilon^2_1). 
\end{align*}
Similarly, \eqref{eq:1changem_t}, \eqref{eq:changezt}, \eqref{eq:bounds}  and \eqref{eq:probabilitiesapprox}  imply,
\begin{align*}
\mathbb{E}(2\bm_{t+1}-2\bm_t|\cH_t,\cE_t) = &-2 - \bigg(2-\frac{2-2p_t^2}{2-p_t^2}\bigg)
\frac{2 |Z_i^t|}{ 2m_t}\cdot 2 + O(\epsilon^2_1)
\\&= -\frac{4-2p_t^2}{2-p_t^2} - \frac{2}{2-p_t^2}(1-p_t)\cdot 2 + O(\epsilon^2_1) =-\frac{8-4p_t-2p_t^2}{2-p_t^2}+ O(\epsilon^2_1).
\end{align*}
For \eqref{eq:changem_t} Lemma \ref{lem:probei} implies that conditioned on $\cH_t,\cE_t$ w.h.p. line 8 of \gr is executed $(1+O(\e_1^2))p_t^2\log^2 n/(2-2p_t^2)$ times during the time interval $[t,t+\log^2 n)$. Thereafter w.h.p. there exists at most two integers $i\in [t,t+\log^2 n)$ such that at step $i$ the lines 5-6 of \gr are executed and an edge whose weight is among the $\log^2 n$ largest is chosen. Hence if  $w(e)=wmax_{t+\log^2n}$ for some edge $e\in E(G_t)$ then $e$ has the $k^{th}$ largest weight among the edges of $E(G_t)$ for some $k=(1+O(\e_1^2))p_t^2\log^2 n/(2-2p_t^2)$ w.h.p. Therefore $b=\mathbb{E}(wmax_{t+\log^2 n}-wmax_t|\cH_t, \cE_t)$ satisfies, 
$$\frac{(1+O(\e_1^2))p_t^2\log^2 n}{m_t(2-2p_t^2)}=\frac{k-1}{m_t}=
\int_{wmax_t-b}^{wmax_t}\frac{e^{-x}}{1-e^{-wmax_t}}dx  = \frac{e^{b}-1}{e^{wmax_t}-1}=\frac{b+O(b^2)}{e^{wmax_t}-1} 
$$ 
and \eqref{eq:changemax} follows. Finally, the expected value of $Exp_{\leq wmax_t}(1)$ is 
$$\frac{\int_{x=0}^{wmax_t}xe^{-x}}{1-e^{-wmax_t}}=\frac{[-xe^{-x}-e^{-x}]_0^{wmax_t}}{1-e^{-wmax_t}}=\frac{1-e^{-wmax}(wmax+1)}{1-e^{-wmax_t}}.$$
Thus,
\begin{align}\label{eq:1changeWt}
\mathbb{E}(W_{t+1}-W_t|\cH_t, \cE_t) =(1-\mathbb{I}(\z_t>0)) wmax_t  +\mathbb{I}(\z_t>0))\cdot \frac{1-e^{-wmax}(wmax+1)}{1-e^{-wmax_t}}+ o(1).
\end{align}
By considering \eqref{eq:changezt} and \eqref{eq:bounds2}, \eqref{eq:changeWt} implies,
\begin{align*}
&\mathbb{E}(\bW_{t+1}-\bW_t|\cH_t) 
\\&=
\mathbb{E}\bigg(\sum_{i=t}^{t+\log^2 n-1}\frac{\mathbb{I}(\z_t=0)}{\log^2 n}wmax_i +\frac{\mathbb{I}(\z_t>0)}{\log^2 n}\frac{1-e^{-wmax}(wmax+1)}{1-e^{-wmax_t}}+ O(\epsilon^2_1) \bigg|  \cH_t,\cE_t\bigg)
\\&= \frac{p_t^2}{2-p^2_t} wmax_t +\bfrac{2-2p_t^2}{2-p_t^2}\frac{1-e^{-wmax}(wmax+1)}{1-e^{-wmax_t}}+ O(\epsilon^2_1)
\\&= \frac{p_t^2}{2-p^2_t} \bfrac{\bwmax_t\log^2 n}{n} +\bfrac{2-2p_t^2}{2-p_t^2} \frac{1-e^{-\bfrac{\bwmax_t\log^2 n}{n}}\bigg(\frac{\bwmax_t\log^2 n}{n}+1\bigg)}{1-e^{-\bfrac{\bwmax_t\log^2 n}{n}}}
+ O(\epsilon^2_1).
\end{align*}
\end{proof}

We are almost ready to apply Theorem \ref{thm:diff}. For that define,
\begin{align*}
    \cD=\bigg\{(t,y,m,maxw,W)&: 0\leq t\leq 1,\hspace{5mm}10^{-13}\leq y\leq 1.1, \hspace{5mm} 2.5y\leq 2m\leq 3.1,
    \\& 0\leq maxw\leq 21 \text{ and } 0\leq W\leq 21\bigg\}.
\end{align*}
Consider the system of differential equations in variable $x \in [0,1]$ with functions $y=y(x),m=m(x), maxw=maxw(x)$ and $W=W(x)$ given by (with $p(x)=3y(x)/2m(x)$)
 \begin{align}\label{eq:Y}
y'(x)= - \frac{2p(x)(2-p(x))}{2-p^2(x)}, 
\end{align}
\begin{align}\label{eq:m}
2m'(x)= -\frac{8-4p(x)-2p^2(x)}{2-p^2(x)},
\end{align}
\begin{align}\label{eq:max}
maxw'(x)= -\frac{(e^{maxw(x)}-1) p^2(x)}{ m(x)(2-p^2(x))},
\end{align}
and
\begin{align}\label{eq:Wt}
W'(x) =  \frac{p^2(x)maxw(x)}{2-p^2(x)} +\bfrac{2-2p^2(x)}{2-p^2(x)} \frac{1-e^{-maxw(x)}(maxw(x)+1)}{1-e^{-maxw(x)}}
\end{align}
with initial condition $$(y(0),m(0),maxw(0),W(0))=(1,1.5,20,0).$$

\begin{lemma}\label{lem:initialconditionstrend}[Initial condition] W.h.p. the following hold.
$$\max\set{\bigg|y(0)-\frac{\bY_0}{\frac{n}{\log^2 n}} \bigg|, \bigg|m(0)-\frac{\bm_0}{\frac{n}{\log^2 n}}\bigg|, \bigg|maxw(0)-\frac{\bwmax_0}{\frac{n}{\log^2 n}}\bigg|, \bigg|W(0)-\frac{\bW_0}{\frac{n}{\log^2 n}}\bigg|} \leq \e_1.$$
\end{lemma}

\begin{proof}
In the event high probability event, $\cE_0$ \eqref{eq:bounds} gives, 
\begin{align*}
    &\max\set{\bigg|y(0)-\frac{\bY_0}{n/\log^2 n} \bigg|, \bigg|m(0)-\frac{\bm_0}{n/\log^2 n}\bigg|} \leq  \max\set{\bigg|y(0)-\frac{|Y^{\tau_1}|}{n} \bigg|, \bigg|m(0)-\frac{m_0}{n}\bigg|}
    \\&\leq \frac{\z_0 + 4\tau_1 \Delta(G)+\e_1^3n}{n}=\e_1 
\end{align*}
Thereafter $W(0)=\bW_0/(n/\log^2 n)=0$. Finally, $|maxw(0)- \bwmax_0/({n}/{\log^2 n})| \geq \e_1^3$ only if fewer that $\tau_1=O(n^{0.95})$ edges $e$ of $G_0$ have weight $w(e)$ in $[20-\e^3_1,20]$. This event occurs with probability $o(1)$.
\end{proof}

Now, for every $(y,m,maxw,W)\in \cD$ we have that $0\leq 3y/2m \leq 1.1$, $2m \geq 2.5y\geq 10^{-10}$ and $0\leq maxw \leq 21$. Therefore,
$y,m,maxw,W$ are $L$-Lipschitz continuous on $\cD$ for some $L\in[1,\infty)$. Thereafter with $\sigma=1-10^5$
we have that $(t,y_1(t),y_2(t),...,y_k(t))$ has $\ell^\infty$ distance at least $3e^L\e_1\leq 10^{-10}$ from the boundary of $\cD$ for all $t\in (0,\sigma)$ provided that $e_1$ is sufficiently small. Here we are using that $y(x),m(x),maxw(x)$ are decreasing and $y(\sigma'), m(\sigma'), maxw(\sigma') \geq 10^{-9}$ with $\sigma'=1-10^{-6}$ while $W(x)$ is increasing and bounded by $20$ on $[0,1]$.  

Therefore, Theorem \ref{thm:diff} implies that if $\epsilon_1$  is such that $3 e^L \epsilon_1 < 10^{-10}$ and $O(\epsilon_1^2)L^{-1}< \epsilon_1$ (the $O(\e_1^2)$ term corresponds to the maximum error term in the Trend-Hypothesis -see Lemma \ref{lem:trend}) 
then w.h.p.
\begin{equation}\label{eq:estimate}
w(M_{\sigma n}) \geq \alpha n > 1.186n \hspace{5mm} \text{ where }\hspace{5mm} \alpha= W(1-10^{-5})-10^{-10},
\end{equation}
$wmax_{\sigma n} \geq maxw(\sigma)-10^{-10}\geq 1.4$ and $2m_t\leq 2m(\sigma n)+10^{-10} \leq 4\cdot 10^{-5}n$.
Part (b) of lemma \ref{lem:performance} implies that $|M_\tau\setminus M_{\sigma n}|= (10^{-5}+o(1))n$. 
$\beta>0$ is chosen such that if we let $E'$ be a set of $2\cdot 10^{-5}n$ many  $Exp_{\leq 1.4}(1)$ independent random variables then the $0.9\cdot 10^{-5}n$ smallest ones sum to a number larger than $\beta n$ w.h.p.
Thus 
$$w(M_{\tau})\geq w(M_{\tau_2})+ (w(M_\tau)-w(M_{\tau_2}))\geq (\alpha+\beta)n.$$

Unfortunately we were not able to find an analytic solution of the system of differential equations \eqref{eq:Y}-\eqref{eq:Wt}. Instead we solved it and calculated $W(1-10^{-5})$ numerically using Mathematica \cite{mathematica}. The corresponding code is located at Appendix \ref{sec:app:code}.

\subsection{Proof of Lemma \ref{lem:performance} (d)}\label{subsec:(c)}
Let $V_2'$ be the number of vertices that are incident to $Y_2$ via an edge in $M_\tau$. As every vertex in $Y_2$ has exactly $2$ neighbors and those lie in $V_2$, part (b) of Lemma \ref{lem:performance} implies that $|V_2\setminus V_2'|=O(n^{0.9})$. Thereafter for $\tau_1\leq t\leq \tau'$, part  (c) of Lemma \ref{lem:performance} implies that 
$$\Pr(u_t\in V_2')= \frac{2|V_2'\cap Z_2^t|}{2m_t}\geq \frac{0.5}{n^{0.05} \log^{-10}n}.$$
Observe that a path spanned by $M_\tau$ spans a vertex in $Y_2$ if it has an interior vertex in $V_2'$. Now let $\cP$ be the set of paths induced by $M_{\tau}$ of length $n^{0.06}$ with more than $n^{0.06}/2$ edges in $M_{\tau'}$. For $P\in \cP$, $w\in V(P)$ and $t\leq \tau'$ such that $w$ is not an endpoint of $P$ and $w=v_t$ set $t=index(w)$. Then, $\Pr(\exists P\in \cP: V(P)\cap V_2\neq \emptyset)$ is bounded above by,
\begin{align*}
& \sum_{P\in \cP} \prod_{t=1}^n \mathbb{I}(t\in \{index(w):w\in V(P)\}) \Pr\bigg(u_t\notin V_2'\bigg| \wedge_{j\leq t} \{u_t \notin V_2'\cap V(P)\}\bigg)
\\&\leq n \bigg(1-\frac{0.5}{n^{0.05} \log^{-10}n}\bigg)^{\frac{n^{0.06}}{2}-2}\leq n e^{-(0.5+o(1))n^{0.01}\log^{-10} n}=o(1).
\end{align*}
\qed

\section{Proof of theorems \ref{thm:mainmult} and \ref{thm:weakmult}}\label{sec:proofmain}

We start by by proving the following lemma.
\begin{lemma}\label{lem:last}
Let $G\sim G_{n,m}(\emptyset,\emptyset,[n],\emptyset,\emptyset)$ be such that $m\leq (1.5+\epsilon_1^3) n$. Assign to every edge $e$ of $G$ a weight $w(e)\sim Exp_{\leq 20}(1)$. Then w.h.p. $G$ spans a cycle $C$ with weight $w(C)\geq (\alpha+0.5\beta)n$ where $\alpha$ is given by \eqref{eq:estimate} and $\beta>0$.
\end{lemma}
\begin{proof}
Let $\bx\in \cS_{n,m}(\emptyset,\emptyset,[n],\emptyset,\emptyset)$, $\bx_1=(\bx_1,\bx_2,....,\bx_{2n^{0.9}})$ and $\bx_2=(\bx_{2n^{0.9}+1},\bx_{2n^{0.9}+2},....,\bx_{2m})$. For $i=0,1,2$ let $Y_i=\{v\in [n]:\bd_{\bx_2}(v)=i\}$ and $Y=[n]\setminus (Y_0\cup Y_1\cup Y_2)$. Then conditioned on $Y_0, Y_1, Y$ and the degree sequence of $G_{\bx_1}$ we have that $\bx_2\in \cS_{n,m}(Y_1,Y_2,Y,\emptyset,\emptyset)$
as there is a $1$-$1$ correspondence between elements of $ \cS_{n,m}(Y_1,Y_2,Y,\emptyset,\emptyset)$ and elements of $ \cS_{n,m}(\emptyset,\emptyset,[n],\emptyset,\emptyset)$ whose degree sequence start with $\bx_1$. Given $d_\bx$, $v\in [n]$ belongs to $Y_i$  only if it has degree $3$ and appears in $d_\bx$ exactly $3-i$ times or it has degree  larger than $3$ and appears in $d_\bx$ more than $3-i$ times. Lemma \ref{lem:maxdegree} states that the maximum degree of $d_\bx$ is $\log n$ w.h.p. while $m\leq (1.5+\e_1^3)n$ implies that at least $(1-\e_1^3)$ vertices have degree $3$ in $d_\bx$ and the set of vertices of degree larger than $3$ is incident to at most $4\e_1^3$ edges. Standard calculations imply that $|Y_2|=\Theta(n^{0.95})$, $|Y_0\cup Y_1|=O(n^{0.9})$ and $G_{\bx_1}$ has at most $o(n^{0.91})$ vertices of degree larger than $1$ w.h.p.  We then execute \gr as described at Section \ref{sec:2greedy} and let $M_\tau$ be the output matching. Lemma \ref{lem:performance} applies thus w.h.p. (I)  $w(M_{\tau})\geq(\alpha+\beta)n$, (II) $M_\tau$ has size $n-n^{0.9}$ and induces $O(n^{0.9})$ components and (III) every path induced by $M_\tau$ of size $n^{0.06}$ that contains at least $0.5n^{0.06}$ edges in $M_{\tau'}$ spans a vertex in $Y_2$.

From each cycle spanned by $M_\tau$ remove an edge to get a set $\cP$ of  $O(n^{0.9})$ paths. Split each path $P$ in $\cP$ into into paths of size $n^{0.095}$, plus possibly a smaller path which we then discard. We also discard any paths that contain a subpath of length $n^{0.06}$ that does not span a vertex in $V_2$. (II) and (III) imply that w.h.p. this gives a set $\cP''$ of $(1+o(n))n^{0.905}$ paths of length $n^{0.095}$ that cover $(1+o(1))$ portion of the edges of $M_\tau$, we assign an arbitrary orientation to each path in $\cP''$. Given $\cP''$ and $G_{\bx_1}$ we create the auxiliary diagraph $G_1$ (also found in \cite{ajtai1981longest}) as follows. $G_1$ has a vertex $v_P$ for each path in $P\in \cP''$. For $P_1,P_2 \in \cP''$ there is an arc from $v_{P_1}$ to $v_{P_2}$ if there is an edge from the last $n^{0.03}$ vertices of $P_1$ that lie in $Y_2$ to the first $n^{0.03}$ vertices of $P_2$ that lie in $Y_2$. 

Observe that for a vertex $v_P\in V(G_1)$ the expected -in degree of $v$ is $(1+o(1))n^{0.03}\cdot n^{0.93}/(2n^{0.95})=(0.5+o(1))n^{0.01}.$ Here  $n^{0.03}$ accounts for the first $n^{0.03}$ vertices of $P$ that lie in $Y_2$,  $2n^{0.95}$ accounts for $|V(G_{X_1})|$ and $(1+o(1))n^{0.93}$ accounts for the number of vertices in $Y_2$ that may be associated with an arc going into $v_P$. It is not hard to show that there exists a coupling of $G_1$ with the random diagraph $D\sim D(n'',p' )$ with parameters $n''=|V(G_1)|$, $p'=\frac{100\log n}{|V(G_1)|}$ such that $D\subset G_1$ w.h.p. Therefore w.h.p. $G_1$ is Hamiltonian \cite{frieze1988algorithm}. A Hamilton cycle in $G_1$ corresponds to a cycle $C$ in $G_{\bx_1}\cup M_\tau$ that spans $(1+o(1))|M_\tau|$ edges of $M_\tau$. As the maximum weight of an edge in $G$ is $20$ we have that $w(C)=w(M_\tau)-o(1)\cdot 20=(\alpha+\beta+o(1))n>(\alpha+0.5\beta)n$.
\end{proof}

\textbf{Proof of Theorem \ref{thm:mainmult}}
Let $G\sim G^\cM(n,n/2+s)$ where $s=s(n)$, $s^3 n^{-2}\to \infty$, $\cK(G)$ be the kernel of $G$,  $C_2(G)$ be the 2-core of the giant component of  $G$ and $n_2(G)$ be the number of vertices of $C_2(G)$. Lemma 2.16 of \cite{frieze2016introduction} implies that there exists $\e_0>0$ such that if $G\sim G^\cM(n,n/2+s)$, $s=s(n) \leq \epsilon_0 n/2$ and $s^3 n^{-2}\to \infty$ then (i) the kernel of $G$, has at most $(1.5+\e_1^3)|V(\cK(G))|$ edges and (ii) if $\gamma\in (0,1]$ is such that $|E(\cK(G))|/\gamma=n_2(C)+|E(\cK(G))|$ then $e^{-20}\geq \gamma$. To each edge $e$ of $\cK(G)$ independently assign a weight $w(e)\sim Exp_{\leq 20}(1)$. Then lemmas \ref{lem:kernel}, \ref{lem:loops}, \ref{lem:transferdistribution}, \ref{lem:maxdegree} and \ref{lem:last} imply that $\cK(G)$ spans a cycle $C$ satisfying $w(C)\geq (\alpha+o(1))|V\cK(G)|$ where $\alpha>1.186$ is given by \eqref{eq:estimate} w.h.p. 
This, together with lemmas \ref{lem:kernel} and \ref{lem:kernelexp1} imply that w.h.p. $G$ spans a cycle of length  
$$ \gamma^{-1}(\alpha+0.5\beta)|V(\cK(G))|= \bigg(1+\frac{n_2(C)}{|E(\cK(G))|}\bigg) (\alpha+0.5\beta)|V(\cK(G))| \geq \frac{\alpha \cdot 16s^2}{3n}$$
given that $\e_0$ is sufficiently small. At the last equality we used that $\beta$ is independent of $\e_0$ and that there exists a function $o(\e_0,n)$ that tends to $0$ as $\e_0\to 0$ and $n\to \infty$ such that 
$n_2(C)=(1+o(\e_0,n))8s^2/n$ and $3|V(\cK(G))|=(2+o(\e_0,n))|E(\cK(G))|=o(\e_0,n)\cdot s^2/n$. 

\textbf{Sketch of the proof of Theorem \ref{thm:weakmult}} Let $G_1\sim G^{\cM}(n,n/2+s-o(n/\log n))$ and $G_2^{\cM}\sim G(n, n/\log_2n)$. Then $G_1\cup G_2\sim G^{\cM}(n,n/2+s)$. To prove Theorem \ref{thm:weakmult} we generate the kernel of $G_1$ and then assign to each edge the same weight. Then, using Lemma \ref{lem:performance2} we find a heavy 2-matching $M$ in $\cK(G_1)$ that does not induced many components. We then proceed as in the proof of Lemma \ref{lem:last} and transform $M$ into a cycle $C$ that spans $(1+o(1))$ portion of the vertices of $\cK(G)$. Finally, by appealing to Lemma \ref{lem:kernel}, one can show that $C$ corresponds to a cycle $C'$ in $G$ which w.h.p. has length  $(1+o(1))(1+n_2(G)/|E(\cK(G_1))|) |V(\cK(G_1))|$. Finally Lemma 2.16 of \cite{frieze2016introduction} implies that $|V(\cK(G))|=(1+o(1))|V(\cK(G_1))|,|E(\cK(G))|=(1+o(1))|E(\cK(G))|$ and $n_2(G)=(1+o(1))n_2(G_1)$ w.h.p.

\section{Concluding Remarks}
In this paper we applied the \gr algorithm to the kernel of $G\sim G(n,c/n),c>1$ in order to construct a small set of $w$-heavy vertex disjoint paths that cover a large portion of the vertices of $G$. By sprinkling some random edges on top of those paths we were able to construct a long cycle. This last part of the argument also implies that the problem of determining $L(G)$ is directly related to the problem of covering the vertices of $G$ with vertex disjoint paths of size  $\omega(n)$ for any function $\omega(n)$ that tends to infinity with $n$. In the regime $c=1+\epsilon$, $\epsilon=o(1)$, removing these paths from the kernel of $G$ leaves a set of edges that is ``close" to a perfect matching. This raises the following question.
\begin{question}
Let $G$ be a graph chosen uniformly from all simple $3$-regular graphs on $[n]$. Assign to every edge $e$ of $G$ a weight $w(e)\sim Exp(1)$ independently. What is the minimum weighted  matching of size $n-n^{0.99}$ of $G$. 
\end{question}
It would also be interesting to examine if our method can be extended to finding long paths in random hypergraphs and improve the corresponding lower bounds, with the first (natural) candidate being loose paths in random $3$-uniform hypergraphs  \cite{cooley2021loose}.

\bibliographystyle{plain}
\bibliography{bib}
\begin{appendix}
\section{Proof of Lemma \ref{lem:loops}}\label{sec:app:lem:loops}
\begin{proof}
Call a cycle of $G$ almost bare if it spans at most $2$ vertices of degree larger than $2$. Note that a loop and a multiple edge of multiplicity $k$ of $\cK(G)$ corresponds to $1$ and $\binom{k}{2}$ respectively almost bare cycles of $G$. The expected number of almost bare cycles of $G$ is 
\begin{align*}
    &\sum_{k\geq 1} \binom{n}{k}\frac{(k-1)!}{2} \binom{k}{2} \frac{    \binom{m}{k}k! [n(n+1)/2- (k-2) (n-2)]^{m-k} }{[n(n+1)/2]^{m}}
    \\ &\leq \sum_{k\geq 1} \frac{k}{4} \prod_{i=0}^k\frac{(n-i)(m-i)}{0.5n(n+1)}
    \bigg(1-\frac{(k-2) (n-2)}{n(n+1)/2}\bigg)^{m-k}
    \\&\leq \sum_{k\geq 1} \frac{(1+o(1))km^2}{n^2} \bfrac{(2+o(1))m}{n}^{k-2} e^{-2\sum_{i=0}^k\frac{i}{n}+\frac{i}{n}}
\cdot e^{-\frac{(2+o(1))(k-2)(m-k)}{n}} 
\\& \leq \sum_{k\geq 1} \frac{(4+o(1))km^2}{n^2} \bfrac{(2+o(1))me^{-\frac{2m}{n}}}{n}^{k-2}
\leq \sum_{k\geq 1} \frac{(4+o(1))km^2}{n^2} \cdot 0.4^{k-2}=O(1).
\end{align*}
Markov's inequality implies that $G$ spans at most $f(n)$ loops and multiple edges w.h.p. Hence w.h.p. $\cK(G)$ spans at most $f(n)$ loops and multiple edges.

$\cK(G)$ spans an edge of multiplicity at least 3 only if $G$ spans a pair of vertices that are joined by $3$ paths whose internal vertices are distinct  and have degree 2. Thereafter $\cK(G)$ spans a loop of multiplicity larger than 2 only if $G$ spans a pair of almost bare cycles that have a single vertex in common. Similarly to the calculations above we have that w.h.p. 
$\cK(G)$ does not span an edge of multiplicity larger than 2 or a loop of multiplicity larger than 1 w.h.p.
\end{proof}

\section{Proof of Lemma \ref{lem:kernelexp1}}\label{sec:app:lem:kernelexp1}
\begin{proof}
Define $\g'$ by $(m'/\gamma)-(m'/\gamma)^{2/3}=(m'/\gamma')$. Let $q_i=1$  for $i> C$ and 
\[q_i=\frac{\Pr(Exp_{\leq C}(1)\in (i\gamma'-\gamma' ,i\gamma'])-\Pr(Geom(\g')=  i  )}{\Pr(Exp(1)_{\leq C}\in (i\gamma'-\gamma', i\gamma])}\]
for  $i\in \{1,2,..., C \rfloor\}$. Note that for $i\leq  C $ we have that 
\begin{align*}
    \frac{\Pr(Exp_{\leq C}(1)\in (i\gamma'-\gamma' ,i\gamma'])}{\Pr(Geom(\g')=  i  )} = \frac{\frac{e^{-(i-1)\g'}(1-e^{-\g'})}{1-e^{-C}}}{\g'(1-\g')^{i-1 }} \geq \frac{\frac{e^{-(i-1)\g'}(1-(1-\g'+(\g')^2))}{1-e^{-C}}}{\g'e^{-(i-1)\g}} =\frac{(1-\g)}{1-e^{-C}}\geq 1.
\end{align*}
Hence $q_i\in [0,1]$ for $i\in \mathbb{N}^+$. For $i\in [m']$ generate $Y_i$ and let
\begin{align*}
 X_i'=\begin{cases} 
 \lceil Y_i /\g' \rceil &\text{with probability } 1-q_{\lceil Y_i /\g' \rceil}, 
 \\ Geom_{>  C }(\g') &\text{otherwise}.
 \end{cases}
\end{align*}
Here by $Geom_{>  C  }(\g')$ we denote the $Geom(\g)$ random variable that is conditioned to be larger than $ C $. Note that by construction $X_i'$ are distributed as independent $Geom(\g')$ random variables. In addition $X_i'\geq Y_i$ for $i\in [m']$. Furthermore,  if $1-\Pr(\sum_{i\in [m']} X_i' \leq m'/\gamma)=o( \Pr(\sum_{i\in [m']} X_i = m'/\gamma) )$ then $X_1',X_2',...,X_{m'}'$ and $X_1,X_2,...,X_{m'}$ can be coupled such that $ X_i' \leq  X_i$ for $i\in[m]$ w.h.p. By linearity we have that $\sum_{i\in [m']} X_i'$ has mean $m'/\gamma'$ and variance $(1-\gamma')m'/(\gamma')^2$. Thus, Chebyshev inequality gives,
\[\Pr\bigg(\sum_{i\in [m']} X_i'> \frac{m'}{\g} \bigg)\leq \frac{\big(\bfrac{m'}{\gamma}^{2/3}\big)^2}{\frac{(1-\gamma')m'}{(\gamma')^2}} \leq 2 m^{2/3}\g^{2/3}.\]
In addition, 
\[\Pr\bigg(\sum_{i\in [m']} X_i=k \bigg) = \binom{k-1}{m'-1} (1-\g)^{k -m'} (\g')^{m'}\]
Thus both the mode and the mean of $\sum_{i\in [m']} X_i$ lie in the interval $[\lfloor m'/\g^{-1}\rfloor,\lfloor m'/\g^{-1}\rfloor+1]$ and by Chebyshev inequality we get that  
\[\Pr\bigg(\sum_{i\in [m']} X_i= m'/\gamma\bigg)\geq \frac{ \Pr\bigg(\bigg|\sum_{i\in [m']} X_i- \frac{m'}{\gamma}\bigg|\leq 10\bfrac{m'}{\gamma}^{0.5}\bigg)}{20\bfrac{m'}{\gamma}^{0.5}}= \Omega(m^{-0.5}\g^{0.5})=\omega(m^{-2/3}\g^{2/3}),\] 
as $\lim_{n\to \infty} \g^{-1}(n)m'(n)=\infty$.
\end{proof}

\section{The degree distribution of random multigraphs with conditioned degree sequence}\label{sec:app:config}
We start by giving a second way of generating $G\sim G_{n,m}(Y_1,Y_2,Y,Z_1,Z)$. For $i\in [n]$ and $\lambda>0$ let $d(i)$ be a random variable generated as follows,
\begin{align}\label{eq:degreeslambda}
    d(i)=
    \begin{cases}
    1&\text{ if }i\in Y_1\cup Z_1,
\\    2&\text{ if }i\in Y_2,
 \\   Po_{\geq 2}(\lambda)&\text{ if } i\in Z,
 \\   Po_{\geq 3}(\lambda)&\text{ if } i\in Z,
 \\ 0&\text{ otherwise.}
    \end{cases}
\end{align}
Here by $Po_{\geq 2}(\lambda)$ and $Po_{\geq 3}(\lambda)$ we denote the Poisson random variables with mean $\lambda$ condition on being at least $2$ and $3$ respectively. So for example $\Pr(Po_{\geq 2}(\lambda)=k)=\lambda^k/(k!(e^\lambda-1-\lambda))$ for $k\geq 2$. We let $ G_{n,m}'(Y_1,Y_2,Y,Z_1,Z)$ be the random graph that is obtained by first generating the decree sequence $\{d(i)\}_{i\in [n]}$ as above, accepting it if the sum of the degrees is $2m$ and then generating a random graph with degree sequence $\{d(i)\}_{i\in [n]}$ using the configuration model. That is, for $i\in [n]$ let $S_i$ be a set of $d(i)$ distinguishable points. Then generate a random pairing of the elements in the multiset $\cup_{i\in [n]}S_i$ and set $G$ be the graph on $[n]$ such that for $i,j\in [n]$ the multiplicity of the edge $ij$ in $G$ is equal to the number of pairs with elements in $S_i\cup S_j$, at least one from each set.

Now let $G$ be a multigraph on $[n]$ with $2m$ edges, $l$ loops, $n_k$ edges/loops of multiplicity $k$,$k\geq 2$ whose degree sequence belongs to $\cS_{n,m}(Y_1,Y_2,Y,Z_1,Z)$ and $G'\sim G_{n,m}'(Y_1,Y_2,Y,Z_1,Z)$. Then with $V=Y_1\cup Y_2 \cup Y\cup Z_1\cup Z$,
\begin{align*}
   \Pr(G'=G)&=C \prod_{v\in Y}\Pr(d_{G'}(v)=d_{G}(v))  \prod_{v\in Z}\Pr(d_{G'}(v)=d_{G}(v)) \frac{\prod_{v\in V} d_{G}(v)!}{2^{l}\prod_{k\geq 2} (k!)^{n_k}} \nonumber
   \\& =C \prod_{v\in Y} \frac{\la^{d_G(v)} }{d_G(v)!(e^\la-1-\la-0.5\la^2)}
   \prod_{v\in Z}  \frac{\la^{d_G(v)} }{d_G(v)!(e^\la-1-\la)} \frac{\prod_{v\in V} d_G(v)!}{2^{l}\prod_{k\geq 2} (k!)^{n_k}} \nonumber
\\& =C \frac{ \la^{2m-2|Y_2|-|Y_1|-|Z_1|}\prod_{v\in Y_1\cup Y_2\cup Z_1} d_G(v)!}{(e^\la-1-\la-0.5\la^2)^{|Y|}(e^\la-1-\la)^{|Z|} 2^{l}\prod_{k\geq 2} (k!)^{n_k}}  \nonumber
 =\frac{C'}{2^{l}\prod_{k\geq 2} (k!)^{n_k}}.
\end{align*}
At the calculations above $C$ equals to $\Pr(\sum_{i\in [n]} d(i)=2m)$ over the number of matchings on $[2m]$ and $C'$ is a constant that depends only on $n,m,|Y_1|,|Y_2|,|Y|,|Z_1|$ and $|Z|$. The last expression is also equal to the probability $G$ is assigned by the $G_{n,m}(Y_1,Y_2,Y,Z_1,Z)$ model as there are $m!2^m/(2^{l}\prod_{k\geq 2} (k!)^{n_k})$ sequences $\bx\in \cS_{n,m}(Y_1,Y_2,Y,Z_1,Z)$ such that $G=G_\bx$. 

We now prove Lemma \ref{lem:maxdegree}. Recall it states the following. 

\begin{lemma}
Let $Y_1, Y_2,  Y, Z_1, Z$ be pairwise disjoint subsets of $[n]$ and $m$ be such that  $|Y_1|+2|Y_2|+3|Y|+|Z_1|+2|Z|\leq 2m =O(n)$. Let $G\sim G_{n,2m}(Y_1,Y_2,Y,Z_1,Z)$. Then w.h.p. $\Delta(G)\leq \log n$.  Thereafter if $Y=[n]$ then $G$ is connected w.h.p. 
\end{lemma}
\begin{proof}
Let $\la$ be chosen such that 
\begin{equation*}
    |Y|\mathbb{E}(Po_{\geq 3}(\la)) +|Z|\mathbb{E}(Po_{\geq 2}(\la)) =2m-|Y_1|-2|Y_2|-|Z_1|.
\end{equation*}
It is standard to show that with $d(i)$ defined by \eqref{eq:degreeslambda} one has, $\Pr(\sum_{v\in [n]} d(v)=2m)=\Omega(n^{-0.5})$. Therefore,
\begin{align*}
     \Pr(\Delta(G)\geq \log n)
     &\leq n\Pr\bigg(Po_{\geq 3}(\lambda)\geq \log n\bigg| \sum_{v\in [n]} d(v)=2m \bigg)\leq \frac{O(n^{1.5}) \la^{\log n}}{(\log n)!(e^\la-1-\la-0.5\la^2)}=o(1).
\end{align*}
Now assume that $Y=[n]$. Then $G$ has minimum degree $3$ and maximum degree $\log n$ w.h.p. Conditioned on the degree sequence $\bd$ of $G$ and the event $\Delta(G)\leq \log n$,  $G$ is not connected with probability at most,
\begin{align*}
    &\sum_{1\leq s\leq n-\frac{n}{2\log n}} \sum_{ 0.5m\geq d_s\geq 1.5s} \binom{n}{s} \bfrac{\frac{(2d_s)!}{d_s! 2^{d_s}}\frac{(2(m-d_s))!}{(m-d_s)! 2^{m-d_s}}}{\frac{(2m)!}{m! 2^m}}=     \sum_{1\leq s\leq n-\frac{n}{2\log n}} \sum_{ 0.5m\geq d_s\geq 1.5s} \binom{n}{s} \prod_{i=1}^{d_s}\frac{2d_s-2i+1}{2m-2i+1}
    \\& \leq  \sum_{1\leq s\leq n-\frac{n}{2\log n}} \sum_{ 0.5m\geq d_s\geq 1.5s}
    \prod_{i=0}^{s-1}\frac{n-i}{s-i}
 \prod_{i=0}^{d_s-1}\frac{ d_s-i }{m-i}
 \leq   \sum_{1\leq s\leq n-\frac{n}{2\log n}} \sum_{ 0.5m\geq d_s\geq 1.5s}
 \prod_{i=s}^{d_s-1}\frac{ d_s-i }{m-i}
     \\& \leq   \sum_{1\leq s\leq n-\frac{n}{2\log n}} \sum_{ 0.5m\geq d_s\geq 1.5s}
 \prod_{i=s}^{d_s-1}\bfrac{ s}{0.5m}^{d_s-s}=o(1).
\end{align*}
Thus $G$ is connected w.h.p.
\end{proof}
By using the Chernoff bound to establish concentration on the number of   $Po_{\geq i}(\lambda),\la=O(1)$ independent random variable that are equal to $k\geq i$ among $\ell$ of them one can prove the following lemma.
\begin{lemma}\label{lem:degDistributions}[Lemma 2.2 of \cite{anastos2021k}]
Let $m=\Theta(n)$ and assume that $2m >  |Y_1|+2|Y_2|+3|Y|+|Z_1|+2|Z|$.
Let $G\sim G_{n,2m}(Y_1,Y_2,Y_3,Z_1,Z_2,X)$. Let $z_i$ and $y_i$ be the number of vertices in $Z$ and $Y$ respectively of degree $i\geq 2$ and $i\geq 3$ respectively. 
Let $\lambda$ be the unique positive real number that satisfies 
\begin{equation}\label{eq:lambda}
    |Y|\mathbb{E}(Po_{\geq 3}(\la)) +|Z|\mathbb{E}(Po_{\geq 2}(\la)) =2m-|Y_1|-2|Y_2|-|Z_1|.
\end{equation}
Then w.h.p.,
$$\bigg|z_i-   \frac{ \lambda^i |Z|}{i! (e^\la-1-\la)} \bigg| \leq m^{0.6}  \text{ for } i\geq 2$$
and
$$\bigg|y_i-  \frac{ \lambda^i |Y|}{i! (e^\la-1-\la-0.5\la^2)} \bigg| \leq m^{0.6}  \text{ for } i\geq 3.$$
\end{lemma}

\section{Proof of parts (b)-(c) of Lemma \ref{lem:performance}}\label{sec:app:auxlemmas}
We start by restating Lemma \ref{lem:performance}.

\begin{lemma} 
Let  $Y_0\sqcup Y_1\sqcup Y_2 \sqcup Y$ be a partition of $[n]$ such that $|Y_2|=\Theta(n^{0.95})$, $|Y_0\cup Y_1|=O(n^{0.9})$ and $2m \geq |Y_1|+2|Y_2|+ 3|Y|$. Let $G\sim G_{n,m}(Y_1,Y_2,Y,\emptyset,\emptyset)$ and $V_2$ be the set of neighbors of $Y_2$ in $G$. To each edge $e\in E(G)$ assign a weight $w(e)\sim Exp_{\leq 20}(1)$. Then execute \gr with input $G,w$ and let $M_i,i\geq 0$ be the generated $2$-matchings, $M=M_\tau$ the output matching and $\tau'=(1-\log^{-1}n)n$. There exists $0<\epsilon_1<0.01$ such that if $ m \leq (1.5+\epsilon_1^3)n$ then w.h.p. the following hold.
\begin{itemize}
    \item[(a)] $\sum_{e\in M_\tau} w(e)\geq (\alpha+\beta)n$ where $\alpha>1.186$ is derived from a system of differential equations (see equations \eqref{eq:Y}-\eqref{eq:Wt} and \eqref{eq:estimate}) and $\beta>0 $ is a sufficiently small constant that is independent of $\alpha$.
    \item[(b)] $M_\tau$ has size $n-O(n^{0.9})$ and spans $O(n^{0.9})$ components.
    \item[(c)] $|V_2\cap V(G_t)| \geq 0.5(1-t/n)^{10}|V_2|$ for $t\leq \tau'$.
    \item[(d)] Every path spanned by $M_{\tau}$ of length at least $n^{0.06}$ with more than $n^{0.06}/2$ edges in $M_{\tau'}$ spans a vertex in $Y_2$.
\end{itemize}
\end{lemma}

\subsection{Proof of part (b)}
Let $\tau_3=\min\{t:m_t\leq n^{0.9}\}$. We start by proving the following lemma.
\begin{lemma}\label{lem:zeta}
W.h.p. $\z_t\leq n^{0.61}$ for $\tau_1\leq t\leq \tau_3$.
\end{lemma}
\begin{proof}
Similarly to the derivation of \eqref{eq:1changez} we have that if $n^{0.6}\leq \z_t\leq n^{0.61}$ and $\tau_1\leq t <\tau_3$ then,
\begin{align*}
    \mathbb{E}(\z_{i+1}+\z_i|\cH_i) &\leq   -1-\frac{\z_t}{2m_t}+\sum_{i\geq 2}\frac{i |Z_i^t|}{ 2m_t}\bigg(  \frac{2|Z_2^t|}{2m_t}+2\frac{3|Y_3^t|}{2m_t} \bigg) + O(m_t^{-1}\log^2 n)
    \\&\leq   -1-\frac{\z_t}{2m_t}+(1-p_t)(1+p_t) + O(m_t^{-1}\log^2 n) 
    \leq -\frac{\z_t}{4m_t} = -O(n^{-0.6}).
\end{align*}
Equation \eqref{eq:bounds} states that $|\z_{t+1}-\z_t|\leq 4\Delta(G_0) \leq 4\log n$ w.h.p. Thus the Azuma-Hoeffding inequality implies that $\z_t\leq n^{0.61}$ for $t\in [\tau_1,\tau_3]$ w.h.p.
\end{proof}

For $t\geq 0$ we let $b(t)$ be the number of edges deleted from vertices incident to $Y_1^t\cup Y_2^t\cup Z_1^t$ at the $t^{th}$ execution of the while loop at line 12 of the description of \grs
Note that for $v\in Y^0$ if no edge incident to $v$ is deleted while $v$ lies in $Y_1^t\cup Y_2^t\cup Z_1^t$ for $t\geq 0$ and $v$ is not incident to a multiple edge or loop in $G_0$ then $v$ is matched twice by \grs Thus a lower bound on $M_\tau$ is given by $|Y^0\cup Y_2|-2\sum_{t \geq 0} b(t)-2mult$ where $mult$ is the number of loops and multiple edges in $G$. At step $t$ we have that $b(t)=k$ only if 1, or 2 vertices in $Z$ are matched and those vertices have $k$ neighbors in $D^t$. 
Hence,
\begin{align*}
    \mathbb{E}(b(t)|\cH_t) = O\bfrac{\z_t}{ 2m_t}.
\end{align*}

$\z_t\leq \max\{200\e_1^2 \z_0,\log^2 n\}=O(n^{0.95})$ for $t\leq \tau_1$ if the with high probability event $\cE_0$ occurs. As $b(t)\leq 2\Delta(G_0)\leq 2\log n$ w.h.p. the Azuma-Hoeffding inequality implies that w.h.p. $\sum_{i=1}^{\tau_1} b(t)=O(n^{0.9})$. Similarly as $\z_t\leq n^{0.61}$ for $t\in [\tau_1,\tau_3)$ the Azuma-Hoeffding inequality implies that w.h.p. $\sum_{i=\tau_1}^{\tau_3} b(t)\leq n^{0.9}$. Hence $\sum_{i=0}^{\tau_3}b(t)=O(n^{0.9})$. Finally, Lemma \ref{lem:loops} implies that $mult\leq \log n$ w.h.p. and therefore w.h.p.
$$ M_\tau \geq  |Y\cup Y_2|-O(n^{0.9})\geq  n-|Y_0\cup Y_1|-O(n^{0.9})= n-O(n^{0.9}).$$ 
Now to show that $M_\tau$ spans $O(n^{0.9})$ components it suffices to show that it spans at most $n^{0.9}$ cycles. For $i\geq 0$ there exists a cycle that is spanned by $M_{i+1}$ and not by $M_i$ if $v_i$ is the endpoint of some path $P$ induced by $M_i$ and $u_i$ is the other endpoint of $P$, note that only one such cycle may exist. The probability of this occurring is bounded by $\Delta(G)/m_i\leq  (2\log n)/n^{0.9}$. Therefore the probability that $M_\tau$ spans more than $n^{0.9}$ cycles is bounded above by 
$$o(1) +\Pr\bigg(Bin(n,\frac{2\log n}{n^{0.9}}\bigg)\geq n^{0.9}\bigg)=o(1).$$

\subsection{Proof of part (c)}
Let $V_2'$ be the set of vertices in $V_2$ that at time $\tau_1$ belong to $Z$ and have degree $2$. Standard (by this point) arguments imply that $|V_2'|\geq 0.6|V_2|$. For $t\geq \tau_1$ let $n_2^t=V_2'\cap V(G_t)$. Also let $\cF_t$ be the event that $\z_t\leq n^{0.91}$ and $m_t\geq 2(n-t)-n^{0.91}$. In the event $M_\tau\geq n-O(n^{0.9})$, as at efvery step a single edge edge is added to the matching, we have that
$m_t\geq \tau-t\geq n-t-n^{0.91}$. Thus  Lemma \ref{lem:zeta} and part (b) of Lemma \ref{lem:performance} imply that $\cap_{t=0}^{\tau'}\cF_t$ occurs w.h.p. 

Similarly to the derivation of \eqref{eq:1expchageY} one has that for $t\geq \tau'$ if $n_2^t\geq n^{0.8}$ (used at the second line of the calculations that follow) then,
\begin{align*}
\mathbb{E}(n_2^{t+1} -n_2^t |\cH_t,\cF_t) &\geq - (2-\mathbb{I}(\z_t>0))\bigg( \frac{2n_2^t}{ 2(n-t-n^{0.91})} +  \sum_{i\geq 2}\frac{i |Z_i^t|}{ 2m_t}   \frac{(i-1)2 n_t}{ 2(n-t-n^{0.91})}\bigg)+O(n^{-0.3})
\\& \geq  -\frac{2.1n_2^t}{n-t}\bigg(1+ \sum_{i\geq 2}\frac{i |Z_i^t|}{ 2m_t}\cdot (i-1) \bigg)    \geq  -\frac{2.1n_2^t}{n-t}\bigg(1+ \sum_{i\geq 2}\frac{i(i-1) \la^i }{ i!(e^\la-1-\la)}  \bigg)
\\& \geq  -\frac{2.1n_2^t}{n-t}\bigg(1+ \frac{ \la^2 }{(e^\la-1-\la)}  \bigg)+  \geq -\frac{2.1n_2^t}{n-t}\bigg(1+\frac{ \la^2 }{0.5\la^2}  \bigg) \geq  
-\frac{7n_2^t}{n-t}.
\end{align*}
At the third inequality we used Lemma \ref{lem:degDistributions}. Thereafter w.h.p. for $t>\tau_1$ we have that $|n_2^{t+1}-n_2^t|\leq 2\Delta(G)\leq 2\log n$. Theorem \ref{thm:diff} implies that w.h.p.
 $$n_2^t\geq \bigg(1-\frac{t}{n-\tau_1}\bigg)^8n_2^{\tau_1}+O(n^{0.91})\geq  0.5\bigg(1-\frac{t}{n}\bigg)^8|V_2|$$
 for $\tau_1\leq i\leq\tau'$.

\section{The differential equation method }\label{sec:app:diffmethod}
In this section, we provide a self-contained statement of the differential equation method in a form that is convenient for our purposes. It combines  Theorem 2 of \cite{warnke2019wormald}, Remark 3 of \cite{warnke2019wormald} and Lemma 9 of \cite{warnke2019wormald}. Note that both the ``Trend Hypothesis " and the ``Boundedness Hypothesis" at Theorem \ref{thm:diff} are stated in a less general form as one need to verify them only for $0\leq i \leq T$ where $T$ is some stopping time. This is exactly the setting of Lemma \ref{lem:trend}. 

Suppose we are given integers $a,n$ a bounded domain $\cD \subseteq \mathbb{R}^{a+1}$, functions $(F_k)_{1\leq k\leq a}$ with $F_k: \cD \mapsto \mathbb{R}$ that are $L$-Lipschitz continuous on $\cD$ for some $L>0$, and $\sigma$-fields $\cF_0\subseteq \cF_1 \subseteq ... \subseteq \cF_n$. 
\begin{theorem}[Differential equation method,  \cite{warnke2019wormald}]\label{thm:diff}
Let $(Y_k(i))_{1\leq k\leq a}$ be $\cF_t$-measurable random variables for $0\leq i \leq n$  and $T\leq n$ be a stopping time with respect to the filtration $(\cF_i)_{0\leq i\leq n}$. Assume that we are given a sequence of events $\{\cE_i\}_{0\leq i\leq n}$ such that $\cap_{i=0}^{n} \cE_i$ holds w.h.p.
Furthermore assume that there exist $\delta,\beta,\lambda$ such that for all $i>0$ and $1\leq k\leq a$ the
following conditions hold whenever $(i/n, Y_1(i)/n, ..., Y_a(i)/n)\in \cD$ and $i<T$
\begin{itemize}
\item[(i)](Trend Hypothesis) $|\mathbb{E}[Y_k(i+1)-Y_k(i)|\cF_i, \cE_i]- F_k(i/n, Y_1(i)/n,...,Y_a(i)/n)|\leq \delta$,
\item[(ii)](Boundedness Hypothesis) $|Y_k(i+1)-  Y_k(i)|\leq \beta $ in the event $\cE_i$ and
\item[(iii)](Initial condition)  $\max_{1\leq k\leq a}|Y_k(0)/n-y_k^0|\leq \lambda$ for some 
$(0, y_1^0,y_2^0...,y_k^0)\in \cD$ w.h.p.
\end{itemize}
Let $(y_{k}(t))_{1\leq k\leq a}$ be the unique  solution to the system of differential equations 
$$y'_k(t)=F_k(t,y_1(t),...,y_k(t))\text{ with } y_k(0)=y_k^0 \text{ for }1\leq k\leq a,$$
and $0\leq \sigma=\sigma(y_1^0,y_2^0,...,y_k^0)\in [0,1]$  be such that $(t,y_1(t),y_2(t),...,y_k(t))$ has $\ell^\infty$ distance at least $3e^{L}\lambda$ from the boundary of $\cD$ for all $t\in (0,\sigma)$. Finally let $R=\max\{\max_{1\leq k\leq a} \max_{x\in \cD}|F_k(x)|,1\}$. Then, whenever $\la\geq \delta L^{-1} +R/n$ we have that w.h.p. the following bound holds,
\begin{equation}\label{eq:diffeqmethod}
     \max_{0\leq t\leq \sigma n}\max_{1\leq k\leq a}\bigg|Y_k(i)-y_k\bfrac{i}{n}n\bigg|\leq 3e^{L}\lambda n.
\end{equation}
\end{theorem}

\section{Mathematica output}\label{sec:app:code}
\includepdf[pages=-]{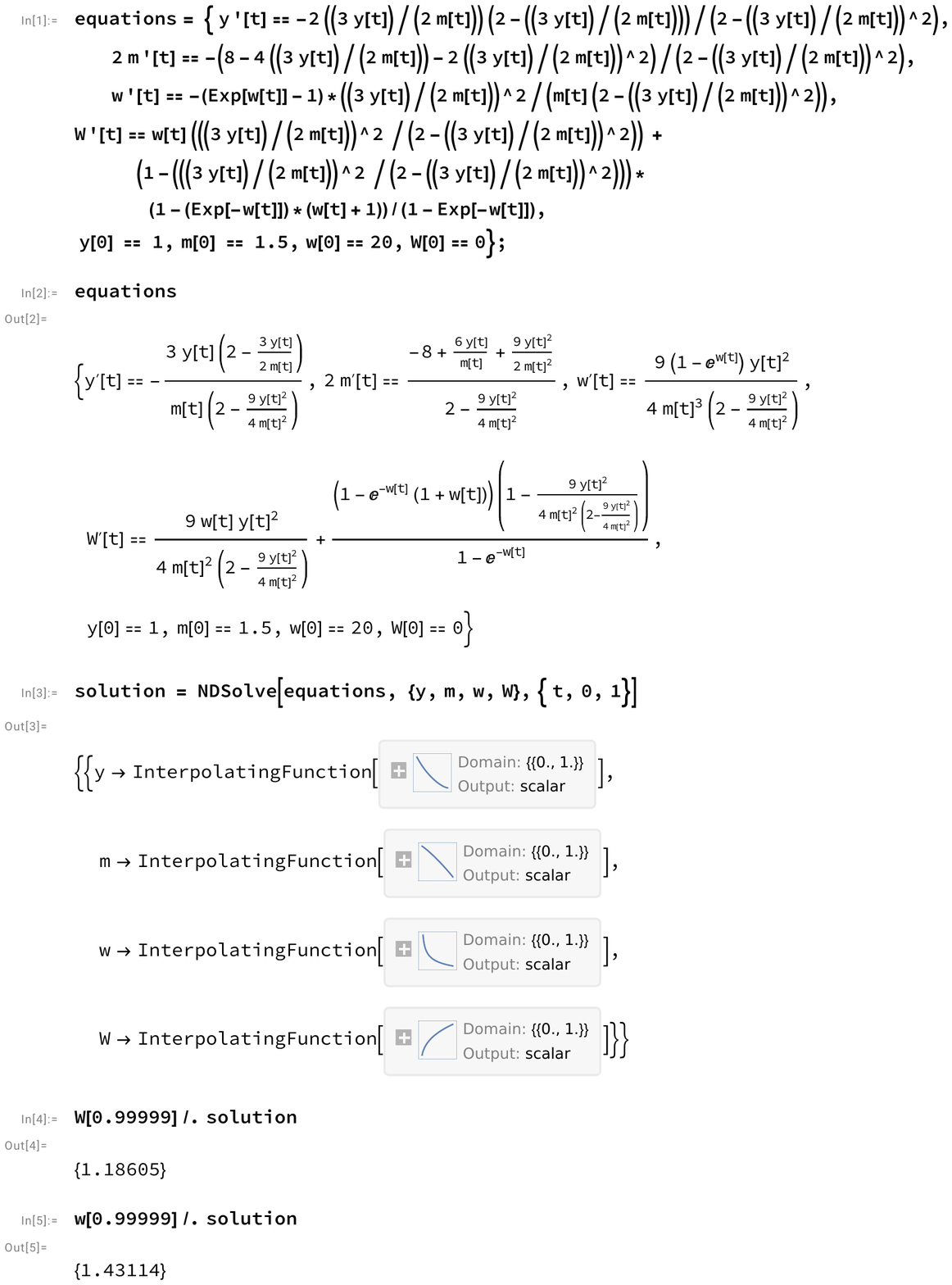}

\end{appendix}

\end{document}